\documentclass[3p,review]{elsarticle}

\usepackage[utf8]{inputenc}
\usepackage{float}
\usepackage{caption}
\usepackage{graphicx}
\usepackage{bm}
\usepackage[svgnames]{xcolor}
\usepackage{xparse}
\usepackage{tikz}
\usepackage{subcaption}
\usetikzlibrary{shapes,arrows,positioning}
\usetikzlibrary{calc}
\usepackage{amsmath}
\usepackage{amsfonts}
\usepackage[binary-units=true]{siunitx}
\usepackage{booktabs, multirow}
\usepackage{mathtools}
\usepackage{amsthm}
\numberwithin{equation}{section}
\usepackage{todonotes}
\usepackage{pgfplots} 
\usepackage{pgfplotstable} 

\pgfplotsset{compat=1.17}
\usepackage[pageanchor]{hyperref}
\usepackage[capitalise]{cleveref}
\usepackage{nomencl}
\usepackage{multicol}
\usepackage{etoolbox}
\usepackage{adjustbox}
\usepackage{mathrsfs}
\usepackage[normalem]{ulem}
\usepackage{pifont}
\usepackage{soul}

\usepackage[shortlabels]{enumitem}
\usepackage[nofillcomment,linesnumbered,vlined,boxed,commentsnumbered]{algorithm2e}

\newtheorem{theorem}{Theorem}
\newtheorem{proposition}[theorem]{Proposition}

\journal{TBA}

\biboptions{authoryear}

\newcommand{\AbbrIwa}{SPACES}
\newcommand{\AbbrIlpRef}{ILP-REF}
\newcommand{\AbbrIlpOur}{ILP-SPACES}
\newcommand{\AbbrBAB}{B\&B-SPACES}

\newcommand{\AbbrPreProc}{P-P}
\newcommand{\AbbrObjective}{\textit{ub}}
\newcommand{\AbbrLb}{\textit{lb}}

\newcommand{\AbbrTime}{\textit{t}}
\newcommand{\AbbrTimeLimit}{TLR}

\newcommand{\DataSingleOff}{NOSBY}
\newcommand{\DataMultiSby}{TWOSBY}

\newcommand{\Second}[1]{\SI{#1}{\second}}

\newcommand{\Optimum}[1]{\bfseries #1}
\newcommand{\TabColSep}{0.7cm}

\newcommand{\DefTerm}[1]{\emph{#1}} 
\newcommand{\AsymCompUpper}[1]{\mathcal{O}(#1)}

\newcommand{\IdxAnother}[1]{#1^{\prime}} %

\DeclareDocumentCommand \SetJobs {o} {
  \IfNoValueTF{#1} {
        \mathcal{J}
  }{
        \mathcal{J}_{#1}
  }
}
\newcommand{\SetJobsPartial}{\widehat{\SetJobs}}
\newcommand{\SetJobsLb}{\SetJobs^{\textnormal{lb}}}
\newcommand{\SetIntervals}{\mathcal{I}} 
\DeclareDocumentCommand \SetProcIntervals {o} {
  \IfNoValueTF{#1} {
        \SetIntervals^{\StateProc}
  }{
        \SetIntervals^{\StateProc}_{#1}
  }
}
\newcommand{\SetStates}{\mathcal{S}} 
\newcommand{\SetIntNonNeg}{\mathbb{Z}_{\geq 0}} 
\newcommand{\SetIntPos}{\mathbb{Z}_{> 0}} 
\newcommand{\SetVertTrans}{V^{\textnormal{is}}} 
\newcommand{\SetEdgesTrans}{E^{\textnormal{is}}} 

\newcommand{\Job}[1]{J_{#1}} 
\newcommand{\JobPart}[1]{J^{1}_{#1}} 
\newcommand{\JobGCD}[1]{J^{\gcd}_{#1}} 
\newcommand{\Interval}[1]{#1}

\newcommand{\IdxJob}{j}
\newcommand{\IdxInterval}{i}
\newcommand{\IdxState}{s}
\newcommand{\NumJobs}{n}
\newcommand{\NumIntervals}{h}
\newcommand{\NumStates}{\vert \SetStates \vert}

\newcommand{\ProcTime}[1]{p_{#1}}

\newcommand{\EnergyCost}[1]{c_{#1}}
\newcommand{\EnergyCostVector}{\bm{c}}

\newcommand{\TotalEnergyCost}[1]{c^{\text{TEC}}(#1)}  

\newcommand{\TransTimeSymbol}{T}
\DeclareDocumentCommand \TransTime {o} {
  \IfNoValueTF{#1} {
        \TransTimeSymbol
  }{
        \TransTimeSymbol(#1)
  }
}

\newcommand{\TransPowerSymbol}{P}
\DeclareDocumentCommand \TransPower {o} {
  \IfNoValueTF{#1} {
        \TransPowerSymbol
  }{
        \TransPowerSymbol(#1)
  }
}

\newcommand{\ProblemNotation}{1,\text{TOU} \vert \, \text{states} \, \vert  \text{TEC}}

\newcommand{\StateOff}{\textup{\texttt{off}}}
\newcommand{\StateIdle}{\textup{\texttt{idle}}}
\newcommand{\StateProc}{\textup{\texttt{proc}}}

\newcommand{\SolVecStarts}{\bm{\sigma}}
\newcommand{\SolStarts}[1]{\sigma_{#1}}
\newcommand{\SolVecTrans}{\boldsymbol{\Omega}}
\newcommand{\SolTrans}[1]{\Omega_{#1}}

\newcommand{\IntInterval}[2]{\{#1, \dots, #2\}}

\newcommand{\VertTrans}[2]{v^{\textnormal{is}}_{#1,#2}}   

\newcommand{\WeightTransSymbol}{w^{\textnormal{is}}}
\DeclareDocumentCommand \WeightTrans {o} {
  \IfNoValueTF{#1} {
        \WeightTransSymbol
  }{
        \WeightTransSymbol(#1)
  }
}

\newcommand{\PathCostTransSymbol}{l^{\textnormal{is}}}
\DeclareDocumentCommand \PathCostTrans {o} {
  \IfNoValueTF{#1} {
        \PathCostTransSymbol
  }{
        \PathCostTransSymbol(#1)
  }
}

\newcommand{\OptCostTransSymbol}{c^{\star}}
\DeclareDocumentCommand \OptCostTrans {o} {
  \IfNoValueTF{#1} {
        \OptCostTransSymbol
  }{
        \OptCostTransSymbol(#1)
  }
}

\newcommand{\OptPathTransSymbol}{\SolVecTrans^{\star}}
\DeclareDocumentCommand \OptPathTrans {o} {
  \IfNoValueTF{#1} {
        \OptPathTransSymbol
  }{
        \OptPathTransSymbol(#1)
  }
}

\newcommand{\FirstOn}{{\textit{early}}} 
\newcommand{\LastOn}{{\textit{late}}} 

\newcommand{\JobCost}[2]{c^{\text{(job)}}_{#1, #2}} 

\SetKwProg{Fn}{Function}{:}{}
\SetKwFunction{KwIwa}{\AbbrIwa{}}
\SetKw{Continue}{continue}
\SetKw{Break}{break}
\SetKw{True}{TRUE}
\SetKw{False}{FALSE}

\DeclareMathOperator{\CmdCpPack}{\textsc{Pack}}

\DeclareDocumentCommand \CpJobsSeq {o} {
  \IfNoValueTF{#1} {
        \pi
  }{
        \pi_{#1}
  }
}

\newcommand{\FixSeqPos}{\ell}
\newcommand{\FixSeqSymbol}{\pi}
\newcommand{\FixSeqPartialSymbol}{\widehat{\FixSeqSymbol}}
\newcommand{\FixSeqPartialLbSymbol}{\FixSeqPartial^{\textnormal{lb}}}

\DeclareDocumentCommand \FixSeq { o } {
    \IfNoValueTF {#1} {
        \FixSeqSymbol
    }{
        \FixSeqSymbol(#1)
    }
}
\DeclareDocumentCommand \FixSeqPartial { o } {
    \IfNoValueTF {#1} {
        \FixSeqPartialSymbol
    }{
        \FixSeqPartialSymbol(#1)
    }
}
\DeclareDocumentCommand \FixSeqPartialLb { o } {
    \IfNoValueTF {#1} {
        \FixSeqPartialLbSymbol
    }{
        \FixSeqPartialLbSymbol(#1)
    }
}

\newcommand{\IdxProcSequence}{k}
\newcommand{\ProcSequenceSymbol}{\Lambda}

\DeclareDocumentCommand \ProcSequence { o } {
    \IfNoValueTF {#1} {
        \ProcSequenceSymbol
    }{
        \ProcSequenceSymbol_{#1}
    }
}

\newcommand{\SetVertTec}{V^{\textnormal{ji}}}
\newcommand{\SetEdgesTec}{E^{\textnormal{ji}}}
\newcommand{\WeightTecSymbol}{w^{\textnormal{ji}}}
\DeclareDocumentCommand \WeightTec {o} {
  \IfNoValueTF{#1} {
        \WeightTecSymbol
  }{
        \WeightTecSymbol(#1)
  }
}
\newcommand{\VertTecSource}{v^{\textnormal{ji}}_{0}}
\newcommand{\VertTecSink}{v^{\textnormal{ji}}_{\NumJobs+1}}
\newcommand{\VertTec}[2]{v^{\textnormal{ji}}_{#1,#2}}   

\newcommand{\ProblemBinPack}{\mathcal{P}_{\text{Bin-Pack}}}
\newcommand{\ProblemBinFind}{\mathcal{P}_{\text{Bin-Find}}}

\definecolor{ClrLightgray}{RGB}{220,220,220}
\definecolor{Clr1}{RGB}{199,229,242}
\definecolor{Clr2}{RGB}{95,171,182}
\definecolor{Clr3}{RGB}{187,217,118}
\definecolor{Clr4}{RGB}{244,209,73}
\definecolor{Clr5}{RGB}{235,121,107}

\newcommand{\xmark}{\ding{55}}%

\newenvironment{Example}[1]
{
  \textbf{Example #1:}\begin{itshape}%
}%
{\end{itshape}}
\newenvironment{ExampleCont}[1]
{
  \textbf{Example #1 (continued):}\begin{itshape}%
}%
{\end{itshape}}

\begin{document}

\begin{frontmatter}
\title{Green Scheduling with Time-of-Use Tariffs and Machine States: Optimizing Energy Cost via Branch-and-Bound and Bin Packing Strategies}

\author[addressCIIRC]{Ondřej Benedikt}
\ead{ondrej.benedikt@cvut.cz}

\author[addressCIIRC]{István Módos}
\ead{istvan.modos@cvut.cz}

\author[addressCIIRC]{Antonin Novak\corref{mycorrespondingauthor}}
\cortext[mycorrespondingauthor]{Corresponding author}
\ead{antonin.novak@cvut.cz}

\author[addressCIIRC]{Zdeněk Hanzálek}
\ead{zdenek.hanzalek@cvut.cz}

\address[addressCIIRC]{Czech Institute of Informatics, Robotics and Cybernetics, Czech Technical University in Prague, Czech Republic}

\begin{abstract}
This paper presents a branch-and-bound algorithm, enhanced with bin packing strategies, for scheduling under variable energy pricing and power-saving states. The proposed algorithm addresses the $\ProblemNotation$ problem, which involves scheduling jobs to minimize total energy cost (TEC) while considering time-of-use (TOU) electricity prices and different machine states (e.g., processing, idle, off). Key innovations include instance pre-processing for rapid lower bound calculations, a novel branching scheme combined with initializations, a block-finding primal heuristic, and a tighter lower bound for jobs with non-coprime processing times. These enhancements result in an efficient algorithm capable of solving benchmark instances with real energy prices with 200 jobs more than 100~times faster than existing state-of-the-art methods.
\end{abstract}

\begin{keyword}
Scheduling \sep Single machine \sep Time of Use \sep Machine states \sep Variable energy costs \sep Total energy cost minimization \sep Branch and Bound.
\end{keyword}

\end{frontmatter}




\section{Introduction}
Optimizing the efficient use of energy in manufacturing facilities is crucial in today's world for several reasons. First, energy costs are rising continuously, and producing products and providing services is becoming more expensive. Therefore, businesses must find ways to lower their energy bills to remain competitive and increase profitability. This is especially important for the automotive or glass and steel industries, which rely on furnaces~\citep{BENEDIKT2021105167}, industrial robots~\citep{BUKATA201952}, and similar machines that require a lot of energy. 
Second, energy consumption is a major contributor to greenhouse gas emissions, which are the leading cause of climate change. By reducing energy consumption, businesses can lower their carbon footprint and help mitigate the effects of climate change.
Moreover, some governments, e.g., EU member states, are implementing regulations and incentives to encourage businesses to reduce their carbon emissions~\citep{maris2021green}. 
Further, rising energy costs, utility companies, and policymakers incentivize businesses to participate in demand-response programs and to plan their operations according to the use of their high-consumption machines to hedge against challenges in the energy market~\citep{eid2016time}.

With demand response programs such as Time-Of-Use (TOU)~\citep{CHEN2019900,2020:hung} tariffs and Real-Time Pricing (RTP)~\citep{2019:abikarram}, businesses can lower their costs by rescheduling their production to hours with cheaper energy or by changing the speeds of the machines. 
Furthermore, efficient use of the alternative processing modes of the machines or turning them off altogether decreases energy consumption. Nevertheless, the turn-on of a machine also requires a considerable amount of energy; thus, keeping the machine idling is sometimes the right choice.

Therefore, production scheduling with variable energy pricing and power-saving modes represents a very complex optimization problem that requires greater investigation.


To address these challenges, we study a single-machine scheduling problem with time-of-use tariffs defining the cost of energy in each time interval.
The scheduling resource is stateful and is described by a transition diagram, defining the duration and consumption of transits between off, processing, and idle states, as well as the consumption in these states.
The goal is to process the set of jobs during processing states so that the total energy cost (TEC) consumed by job processing and transitions between the machine states is minimized.
The problem denoted in an extended three-field notation as $\ProblemNotation$ was originally introduced in \cite{2014:Shrouf}. 
Later on, it was shown to be strongly $\textsf{NP}$-hard in \cite{aghe2019toustates}.

\subsection{Paper Contributions and Outline}\label{sec:contribution}
We revisit the problem $\ProblemNotation$ with new insights with the goal of developing a scalable, exact branch-and-bound algorithm called \AbbrBAB{} that achieves state-of-the-art results, outperforming the former mixed-integer linear programming models both of~\cite{2014:Shrouf} and \cite{2020a:Benedikt} by two orders of magnitude.
Among these insights, recognizing the bin-packing structure of the problem~\citep{grus_icores24} and utilizing it in different aspects of \AbbrBAB{} proves to be the major source of increased efficiency.

More specifically, the main contributions are:
\begin{enumerate}
    \item a fast branch-and-bound algorithm able to solve instances with hundreds of jobs over more than a thousand time intervals, outperforming the former state-of-the-art method (see \cref{sec:bab}),
    \item employing instance pre-processing for lower bound computation, improving the complexity of the algorithm for a fixed sequence by~\cite{aghe2019toustates} (see \cref{sec:bab-lower-bound}),
    \item proposing a tighter lower bound for non-coprime processing times (see \cref{sec:bab-lb-gcd}),
    \item new bin-packing primal heuristics for computing upper bounds inside the search tree (see \cref{sec:binpack_primal_heur}),
    \item new initialization heuristics based on the \emph{bin-finding} problem (see \cref{sec:bin_finding}).
\end{enumerate}

The rest of the paper is structured as follows.
In \cref{sec:probstat}, we formally define the statement of the scheduling problem under study together with the key notation and symbols used.
The relevant literature surveying the related problem is given in \cref{sec:relwork}.
The two key concepts---the optimal switching problem and computing TEC for a given fixed sequence of the jobs are explained in \cref{sec:spaces} and \cref{sec:fixseq_tec}.
Further algorithmic components, i.e., lower and upper bounds of \AbbrBAB{} algorithm, are described in \cref{sec:bab}.
Finally, we demonstrate the efficiency of the resulting method on the benchmark instances of various parameters in Section~\ref{sec:experiments}.

\section{Problem Statement}\label{sec:probstat}

We consider a scheduling problem whose horizon is divided into a set of $h\in\SetIntPos$ non-overlapping intervals $\SetIntervals = \{1, 2, \dots, \IdxInterval, \ldots, \NumIntervals\}$. 
It is assumed that every interval is one unit long. 
 
The physical representation of the time unit length can be different depending on the desired granularity.
For example, most typically, a single interval will resemble 15-minute or 1-hour long intervals, as the energy markets operate on these scales.  
Scheduling over one thousand intervals resembles operations spanning over more than a week-long horizon, matching the real conditions commonly encountered on the market (such as OTE in the Czech Republic) regarding the forecast and auctions.
For every interval $\Interval{\IdxInterval}$, the \DefTerm{energy cost} $\EnergyCost{\IdxInterval} \in \mathbb{Q}$ is given. The vector of energy costs is denoted by $\EnergyCostVector$, where $\EnergyCostVector = (\EnergyCost{1}, \EnergyCost{2}, \dots, \EnergyCost{\NumIntervals})$.

Note that we assume that the energy cost can be even negative. 
Although this is likely not to be the realistic case with, e.g., natural gas, with electricity, the situation is different. 
For example, due to the increasing amount of power deployed in renewable sources every year, from April to August 2024 in the Czech Republic, there were nearly 50 days with intervals having negative electricity cost---up to 3 times more than the year before.

\begin{figure}[ht]
    \centering
    {
\newcommand{\ImNumInts}{20} 
\newcommand{\ImNumRows}{7} 
\newcommand{\ImGridSize}{0.5} 

\newcommand{\ImIntervalRow}{7}
\newcommand{\ImPriceRow}{6}
\newcommand{\ImProcessingRow}{5}
\newcommand{\ImIdleRow}{4}
\newcommand{\ImOffRow}{3}
\newcommand{\ImTurnOnRow}{1}
\newcommand{\ImTurnOffRow}{2}
\newcommand{\ImScheduleRow}{0}

\begin{tikzpicture}

	\tikzset{cell/.style = {rectangle, minimum width=\ImGridSize cm,	minimum height = \ImGridSize cm, inner sep=0pt,
  text width=\ImGridSize cm, align=center},
	graycell/.style = {rectangle, minimum width=\ImGridSize cm,	minimum height = \ImGridSize cm, fill=ClrLightgray,  inner sep=0pt,
  text width=\ImGridSize cm, align=center},
	desc/.style = {anchor=east, xshift=-0.4 cm}}

	\foreach \x in {0,...,\ImNumInts}
	{
		\foreach \y in {0,...,\ImNumRows}
		{
			\pgfmathsetmacro{\crdX}{(\x-1) * \ImGridSize + \ImGridSize/2};
        		\pgfmathsetmacro{\crdY}{(\y-1) * \ImGridSize + \ImGridSize/2};        		
        		\coordinate (\x-\y) at (\crdX,\crdY);        		
		}
	}

	\foreach \i in {1,...,\ImNumInts} {
		\node[graycell] at (\i-\ImIntervalRow) {\scriptsize  $\Interval{\tiny \i}$};
	}
	
	
	\foreach \i/\price in {1/9, 2/7, 3/9, 4/13, 5/3, 6/11, 7/3, 8/13, 9/6, 10/7, 11/60, 12/4, 13/10, 14/6, 15/9, 16/3, 17/14, 18/0, 19/4, 20/6} {
		\node[graycell] at (\i-\ImPriceRow) {\scriptsize \price};
	}
	
	\foreach \i/\price/\clr in {14/24/Clr1,7/12/Clr1,8/52/Clr1,15/36/Clr1,16/12/Clr1,17/56/Clr1,18/0/Clr1} {
		\node[cell, fill=\clr] at (\i-\ImProcessingRow) {\scriptsize \price};
	}
	
	
	\foreach \i/\price/\clr in {1/0/Clr3,2/0/Clr3, 3/0/Clr3, 4/0/Clr3,10/0/Clr3, 11/0/Clr3, 20/0/Clr3} {
		\node[cell, fill=\clr] at (\i-\ImOffRow) {\scriptsize \price};
	}
	
	\foreach \i/\price/\clr in {5/15/Clr4,6/55/Clr4,12/20/Clr4,13/50/Clr4} {
		\node[cell, fill=\clr] at (\i-\ImTurnOnRow) {\scriptsize \price};
	}
	
	\foreach \i/\price/\clr in {9/6/Clr5,19/4/Clr5} {
		\node[cell, fill=\clr] at (\i-\ImTurnOffRow) {\scriptsize \price};
	}
	
	\node[desc] at (1-\ImIntervalRow) {\scriptsize Interval $\Interval{\IdxInterval}$};
	\node[desc] at (1-\ImPriceRow) {\scriptsize Energy cost $\EnergyCost{\IdxInterval}$};
	\node[desc] at (1-\ImProcessingRow) {\scriptsize Processing};
	\node[desc] at (1-\ImIdleRow) {\scriptsize Idle};
	\node[desc] at (1-\ImOffRow) {\scriptsize Off};
	\node[desc] at (1-\ImTurnOnRow) {\scriptsize Turn-on ($\nearrow$)};
	\node[desc] at (1-\ImTurnOffRow) {\scriptsize Turn-off ($\searrow$)};	
	
	\draw[step=\ImGridSize,black] (0,0) grid (\ImNumInts*\ImGridSize, \ImNumRows*\ImGridSize);
	
	\node at (0,-0.7 cm) {};  
	
	\begin{scope}[transform canvas={yshift=-0.2cm}]
		\node[desc] at (1-\ImScheduleRow) {\scriptsize Schedule};
		
		\foreach \xs/\xe/\lbl/\clr in {
			1/4/$\StateOff$/Clr3,
			5/6/$\nearrow$/Clr4,
			7/8/$\Job{2}$/Clr1,
			9/9/$\searrow$/Clr5,
			10/11/$\StateOff$/Clr3,
			12/13/$\nearrow$/Clr4,
			14/14/$\Job{1}$/Clr1,
			15/18/$\Job{3}$/Clr1,
			19/19/$\searrow$/Clr5,
			20/20/$\StateOff$/Clr3} {
			\draw[fill=\clr] ($ (\xs-0) - 1/2*(\ImGridSize,\ImGridSize) $) rectangle ($ (\xe-0) + 1/2*(\ImGridSize,\ImGridSize) $) node[pos=0.5] {\tiny \lbl};
		}		
	\end{scope}
	
\end{tikzpicture}
}
    \caption{Example of a schedule to illustrate the notation; optimal solution with the objective 342.}
    \label{fig:example-schedule}
\end{figure}

During each interval, the machine operates in one of its \DefTerm{states} $\IdxState \in \SetStates$ or transits from one state to another.
We assume the existence of at least two states: the \DefTerm{off state} $\StateOff \in \SetStates$ and the \DefTerm{processing state} $\StateProc \in \SetStates$.
The machine is assumed to be operating in $\StateOff$ state during the first and the last interval. Also, when a job is processed, the machine must be operating in $\StateProc$ state.

The \DefTerm{transition time function} and the \DefTerm{transition power function} are denoted by $\TransTime : \SetStates \times \SetStates \rightarrow \SetIntNonNeg \cup \{\infty\}$ and \hbox{$\TransPower : \SetStates \times \SetStates \rightarrow \mathbb{Q}_{\geq0} \cup \{ \infty \}$}, respectively.
The symbol $\infty$ is used to denote that the direct transition between states $(\IdxState,\IdxAnother{\IdxState})\in\SetStates\times\SetStates$ does not exist.
The transition from state $\IdxState$ to state $\IdxAnother{\IdxState}$ lasts $\TransTime[\IdxState,\IdxAnother{\IdxState}]$ intervals and has power consumption  $\TransPower[\IdxState,\IdxAnother{\IdxState}]$, which is the constant rate of the consumed energy at every time unit.
Further, we assume that $\TransTime[\IdxState,\IdxState]=1$, meaning that the power consumption of the machine while remaining in state $\IdxState$ for the duration of one interval is $\TransPower[\IdxState,\IdxState]$.
Note that the transition time and power functions are general enough to represent different kinds of manufacturing machines, e.g., those studied in~\citep{2018:Aghelinejad,2020a:Benedikt,2007:mouzon,2014:Shrouf}.

Let $\SetJobs$ be a set of $n$ \DefTerm{jobs} $\{\Job{1}, \Job{2}, \dots, \Job{\IdxJob}, \ldots, \Job{\NumJobs} \}$ to be scheduled.
Each job $\Job{\IdxJob}$ is characterized by its \DefTerm{processing time} $\ProcTime{\IdxJob} \in \SetIntPos$, given in the number of intervals.
The cost of processing a job inside an interval $\Interval{\IdxInterval}$ is the same for all jobs.
Jobs are non-preemptive, and the machine can process at most one job at a time.
All the jobs are available at the beginning of the scheduling horizon.

A \DefTerm{solution} is defined by a pair $(\SolVecStarts, \SolVecTrans)$, where $\SolVecStarts = (\SolStarts{1}, \SolStarts{2}, \dots, \SolStarts{\NumJobs}) \in \SetIntNonNeg^{\NumJobs}$ represents start times of the jobs, and $\SolVecTrans = (\SolTrans{1}, \SolTrans{2}, \dots, \SolTrans{\NumIntervals}) \in (\SetStates \times \SetStates)^{\NumIntervals}$ captures the active state or transition in each interval, where $\SolTrans{\IdxInterval}=(\StateOff,\StateProc)$ represents turning on, $\SolTrans{\IdxInterval}=(\StateProc,\StateProc)$ represents staying in the processing state, $\SolTrans{\IdxInterval}=(\StateProc,\StateIdle)$ represents transition from processing into an idle state, $\SolTrans{\IdxInterval}=(\StateIdle,\StateIdle)$ staying in an idle state, $\SolTrans{\IdxInterval}=(\StateIdle,\StateProc)$ is transition from idle to processing state while $\SolTrans{\IdxInterval}=(\StateProc,\StateOff)$ represents the transition from the processing state to the off state.



The solution is \DefTerm{feasible} if the following four conditions are satisfied:

\begin{enumerate}
    \item the machine processes at most one job at a time,
    
    \item the jobs are processed only when the machine is in $\StateProc$ state,
    \item the machine is in $\StateOff$ state during the first and the last interval, 
    \item transition between states $\IdxState$ and $\IdxAnother{\IdxState}$ takes $\TransTime[\IdxState,\IdxAnother{\IdxState}]$ time intervals.
\end{enumerate}
The total energy cost (TEC) of solution $(\SolVecStarts, \SolVecTrans)$ is  calculated as
\begin{equation} \label{eq:tec}
    \sum\limits_{\Interval{\IdxInterval} \in \SetIntervals} \EnergyCost{\IdxInterval} \cdot \TransPower[\SolTrans{\IdxInterval}],
\end{equation}
where $\TransPower[\SolTrans{\IdxInterval}]$ represents $\TransPower[\IdxState,\IdxAnother{\IdxState}]$ for $\SolTrans{\IdxInterval} = (\IdxState, \IdxAnother{\IdxState})$.
The goal of the scheduling problem is to find a feasible solution such that the total energy cost \eqref{eq:tec} is minimized.

During the first and the last interval, the machine is assumed to be in off state $\StateOff \in \SetStates$ (see 3. above).
The machine has a single processing state $\StateProc \in \SetStates$, which must be active during the processing of the jobs (see 2. above). 
Due to the transition time from/to the initial/last $\StateOff$ state, the machine cannot be in the $\StateProc$ state during (several) early/late intervals.
Hence, we denote the \DefTerm{earliest} and the \DefTerm{latest} interval during which the machine can be in $\StateProc$ state by $\Interval{\FirstOn}$ and $\Interval{\LastOn}$, respectively.
Consequently, the set of \DefTerm{processing intervals} is denoted as \( \SetProcIntervals = \{ \Interval{\FirstOn}, \Interval{\FirstOn + 1}, \dots, \Interval{\LastOn} \}\).
The definition of the \textit{earliest} and the \textit{latest} processing intervals becomes useful later in the problem of the optimal TEC computation for a fixed job sequence in \cref{sec:fixseq_tec}.

\begin{Example}{1}\label{ex:example1}
Let us consider a small example with three jobs, $\SetJobs{} = \{\Job{1}, \Job{2}, \Job{3}\}$ with processing times $\ProcTime{1} = 1$, $\ProcTime{2} = 2$, and $\ProcTime{3} = 4$, respectively. Let the scheduling horizon contain 20 intervals, $\SetIntervals = \{ \Interval{1}, \dots, \Interval{20} \}$, with energy costs $\EnergyCostVector = (9, 7, 9, 13, 3, 11, 3, 13, 6, 7, 60, 4, 10, 6, 9, 3, 14, 0, 4, 6)$. Considering the machine states and transitions, let us use the transition graph proposed in \cite{2014:Shrouf}, which is also depicted in \cref{fig:example-func-power-time}.
The optimal solution to this instance is depicted in \cref{fig:example-schedule}. Its TEC is equal to $342$.

\end{Example}
\begin{figure}[ht]
    \centering
    \begin{tabular}{ccc}
    {
\newcommand{\ImNumCols}{3} 
\newcommand{\ImNumRows}{4} 
\newcommand{\ImGridSize}{0.7} 

\begin{tikzpicture}
	\foreach \x in {0,...,\ImNumCols}
	{
		\foreach \y in {0,...,\ImNumRows}
		{
			\pgfmathsetmacro{\crdX}{(\x-1) * \ImGridSize + \ImGridSize/2};
        		\pgfmathsetmacro{\crdY}{(\y-1) * \ImGridSize + \ImGridSize/2};        		
        		\coordinate (\x-\y) at (\crdX,\crdY);        		
		}
	}
	
	\pgfmathsetmacro{\drawRows}{\ImNumRows - 1};
	
	\draw (0,0) -- (0,\drawRows*\ImGridSize);
	\draw (0,\drawRows*\ImGridSize) -- (\ImNumCols*\ImGridSize,\drawRows*\ImGridSize);
	
	\node at (1-3) {4};
	\node at (1-2) {0};
	\node at (1-1) {5};
	
	\node at (2-3) {0};
	\node at (2-2) {2};
	\node at (2-1) {$\infty$};
	
	\node at (3-3) {1};
	\node at (3-2) {$\infty$};
	\node at (3-1) {0};
	
	\node[anchor=east] at ($(0-3) + (\ImGridSize/3,0)$) {\scriptsize $\StateProc$};
	\node[anchor=east] at ($(0-2) + (\ImGridSize/3,0)$) {\scriptsize $\StateIdle$};
	\node[anchor=east] at ($(0-1) + (\ImGridSize/3,0)$) {\scriptsize $\StateOff$};	
	
	\node at ($(1-4) + (0,-\ImGridSize/5)$) {\scriptsize $\strut \StateProc$};
	\node at ($(2-4) + (0,-\ImGridSize/5)$) {\scriptsize $\strut \StateIdle$};
	\node at ($(3-4) + (0,-\ImGridSize/5)$) {\scriptsize $\strut \StateOff$};
	
	\draw ($(1-3) + (-\ImGridSize/2,\ImGridSize/2)$) -- ($(1-3) + (-3*\ImGridSize/2,\ImGridSize)$);
	\node[anchor=north east] at ($(0-4) + (0,-\ImGridSize/6)$) {$\IdxState$};
	\node[anchor=south west] at ($(0-4) + (-\ImGridSize/5,-\ImGridSize/6)$) {$\IdxAnother{\IdxState}$};
	
	\node at (2-0) {$\strut \TransPower[\IdxState, \IdxAnother{\IdxState}]$};
\end{tikzpicture}
} &
    {
\newcommand{\ImNumCols}{3} 
\newcommand{\ImNumRows}{4} 
\newcommand{\ImGridSize}{0.7} 

\begin{tikzpicture}
	\pgfmathsetmacro{\prepareLayersX}{\ImNumCols+1};
	\pgfmathsetmacro{\prepareLayersY}{\ImNumRows+1};
	\foreach \x in {0,...,\prepareLayersX}
	{
		\foreach \y in {0,...,\prepareLayersY}
		{
			\pgfmathsetmacro{\crdX}{(\x-1) * \ImGridSize + \ImGridSize/2};
        		\pgfmathsetmacro{\crdY}{(\y-1) * \ImGridSize + \ImGridSize/2};        		
        		\coordinate (\x-\y) at (\crdX,\crdY);        		
		}
	}
	
	\pgfmathsetmacro{\crdX}{0*\ImGridSize};
	\pgfmathsetmacro{\drawRows}{\ImNumRows - 1};
	
	\draw (0,0) -- (0,\drawRows*\ImGridSize);
	\draw (0,\drawRows*\ImGridSize) -- (\ImNumCols*\ImGridSize,\drawRows*\ImGridSize);
	
	\node at (1-3) {1};
	\node at (1-2) {0};
	\node at (1-1) {2};
	
	\node at (2-3) {0};
	\node at (2-2) {1};
	\node at (2-1) {$\infty$};
	
	\node at (3-3) {1};
	\node at (3-2) {$\infty$};
	\node at (3-1) {1};
	
	\node[anchor=east] at ($(0-3) + (\ImGridSize/3,0)$) {\scriptsize $\StateProc$};
	\node[anchor=east] at ($(0-2) + (\ImGridSize/3,0)$) {\scriptsize $\StateIdle$};
	\node[anchor=east] at ($(0-1) + (\ImGridSize/3,0)$) {\scriptsize $\StateOff$};	
	
	\node at ($(1-4) + (0,-\ImGridSize/5)$) {\scriptsize $\strut \StateProc$};
	\node at ($(2-4) + (0,-\ImGridSize/5)$) {\scriptsize $\strut \StateIdle$};
	\node at ($(3-4) + (0,-\ImGridSize/5)$) {\scriptsize $\strut \StateOff$};
	
	\draw ($(1-3) + (-\ImGridSize/2,\ImGridSize/2)$) -- ($(1-3) + (-3*\ImGridSize/2,\ImGridSize)$);
	\node[anchor=north east] at ($(0-4) + (0,-\ImGridSize/6)$) {$\IdxState$};
	\node[anchor=south west] at ($(0-4) + (-\ImGridSize/5,-\ImGridSize/6)$) {$\IdxAnother{\IdxState}$};

	\node at (2-0) {$\strut \TransTime[\IdxState, \IdxAnother{\IdxState}]$};
\end{tikzpicture}
} &
    \begin{tikzpicture}[auto,-stealth]
\tikzset{
    state/.style={circle,draw=black,inner sep=1pt, minimum size=0.8cm},
    every edge/.append style={-stealth,thick}
}
    
\node[state] at (1,0.6) (nProc) {\scriptsize $\strut \StateProc$};
\node[state] at (2.5,-0.7) (nIdle) {\scriptsize $\strut \StateIdle$};
\node[state] at (-0.5,-0.7) (nOff)  {\scriptsize $\strut \StateOff$};

\path (nProc) edge[bend left=15] node[above right, pos=0.8] {\scriptsize 0/0} (nIdle);
\path (nIdle) edge[bend left=15] node[below left, pos=0.2] {\scriptsize 0/0} (nProc);

\path (nProc) edge[bend left=15] node[below right, pos=0.8] {\scriptsize 1/1} (nOff);
\path (nOff) edge[bend left=15] node[above left, pos=0.2] {\scriptsize 2/5} (nProc);

\path (nProc) edge[loop above,looseness=4, out=105, in=75] node[left,pos=0.1] {\scriptsize 1/4} (nProc);

\path (nOff) edge[loop below,looseness=4, in=255, out=285] node[right,pos=0.2] {\scriptsize 1/0} (nOff);

\path (nIdle) edge[loop below,looseness=4, in=255, out=285] node[right,pos=0.2] {\scriptsize 1/2} (nIdle);

\node at (0,-2) {};
\end{tikzpicture}  \\
    \end{tabular}
    \caption{Parameters of the transition power function \( \TransPower[\IdxState, \IdxAnother{\IdxState}] \) and transition time function \( \TransTime[\IdxState, \IdxAnother{\IdxState}] \), and the corresponding transition graph, where every edge from $\IdxState$ to $\IdxAnother{\IdxState}$ is labeled by $\TransTime[\IdxState, \IdxAnother{\IdxState}]$/$\TransPower[\IdxState, \IdxAnother{\IdxState}]$.}
    \label{fig:example-func-power-time}
\end{figure}
Note that the transition graph in \cref{fig:example-func-power-time} considers transitions between $\StateProc$ and $\StateIdle$ states to take zero time and to consume zero energy while doing so.  
We argue that this is a rather physically unrealistic case that we follow from~\cite{2014:Shrouf}, but it serves as an example of extreme behavior that can be modeled with transition graphs.  
In general, including an $\StateIdle$ state into the model of a machine may be beneficial depending on the values of transition power and time functions, as demonstrated in the following example.

\begin{Example}{2}
    Consider an example of a machine modeled with two different transition diagrams. 
    The option in \cref{fig:isolines_no_idle} resembles a simplified model of the machine that neglects the possibility of utilizing its $\StateIdle$ state, while the option \cref{fig:isolines_idle} explicitly contains it.
    Each respective plot above the transition diagram displays TEC isolines (i.e., the areas with identical TEC values), depending on the values of the transition power function between $\{\StateOff,\StateProc\}$ states.
    The values are computed by sampling $20\times 20$ pairs of parameter values in the considered range while constructing an optimal schedule for a set of 5 jobs with processing times $\{1,2,2,4,5\}$ under the energy cost profile shown in \cref{fig:isolines_cost}. 
    The example shows (besides overall lower TEC in \cref{fig:isolines_idle}) how introducing an $\StateIdle$ state can change the landscape of TEC isolines, affecting trade-offs of transition power function values between $\StateOff$ and $\StateProc$ states. 

\begin{figure}[ht]
    \centering
    \begin{subfigure}{.5\textwidth}
    \centering
        \begin{tikzpicture}
    \begin{axis}[
            ylabel={$\TransPower[\StateOff,\StateProc]$},
            xlabel={$\TransPower[\StateProc,\StateOff]$},
            height=0.8\textwidth,
            width=0.8\textwidth,
        ]
        \addplot [
            contour prepared,
            contour prepared format=matlab,
                point meta min=528,
                point meta max=630
        ] table {other/contour_data_no_idle.txt};
    \end{axis}
    \tikzset{
    state/.style={circle,draw=black,inner sep=1pt, minimum size=0.8cm},
    every edge/.append style={-stealth,thick}
}
    
\node[state] at (4.9,-2.5) (nProc) {\scriptsize $\strut \StateProc$};
\node[state] at (0.8,-2.5) (nOff)  {\scriptsize $\strut \StateOff$};

\node at (0,-3.7) {};

\path (nProc) edge[bend left=15] node[below right, pos=0.8] {\scriptsize 1/$\TransPower[\StateProc,\StateOff]$} (nOff);
\path (nOff) edge[bend left=15] node[above right, pos=0.2] {\scriptsize 2/$\TransPower[\StateOff,\StateProc]$} (nProc);

\path (nProc) edge[loop below,looseness=4, out=105, in=75] node[left,pos=0.1] {\scriptsize 1/5} (nProc);

\path (nOff) edge[loop below,looseness=4, in=255, out=285] node[right,pos=0.2] {\scriptsize 1/0} (nOff);

    \end{tikzpicture}
      \caption{Machine without $\StateIdle$ state.}
      \label{fig:isolines_no_idle}
    \end{subfigure}%
    \begin{subfigure}{.5\textwidth}
      \centering
        \begin{tikzpicture}
        \begin{axis}[
                ylabel={$\TransPower[\StateOff,\StateProc]$},
                xlabel={$\TransPower[\StateProc,\StateOff]$},
                height=0.8\textwidth,
                width=0.8\textwidth,
            ]
            \addplot [
                contour prepared,
                contour prepared format=matlab,
                point meta min=528,
                point meta max=630
            ] table {other/contour_data_with_idle.txt};
        \end{axis}
        \tikzset{
    state/.style={circle,draw=black,inner sep=1pt, minimum size=0.8cm},
    every edge/.append style={-stealth,thick}
}
    
\node[state] at (3,-2) (nProc) {\scriptsize $\strut \StateProc$};
\node[state] at (5,-3) (nIdle) {\scriptsize $\strut \StateIdle$};
\node[state] at (1,-3) (nOff)  {\scriptsize $\strut \StateOff$};

\path (nProc) edge[bend left=15] node[above right, pos=0.8] {\scriptsize 1/0} (nIdle);
\path (nIdle) edge[bend left=15] node[below left, pos=0.2] {\scriptsize 1/1} (nProc);

\path (nProc) edge[bend left=15] node[below right, pos=0.8] {\scriptsize 1/$\TransPower[\StateProc,\StateOff]$} (nOff);
\path (nOff) edge[bend left=15] node[above left, pos=0.2] {\scriptsize 2/$\TransPower[\StateOff,\StateProc]$} (nProc);

\path (nProc) edge[loop above,looseness=4, out=105, in=75] node[left,pos=0.1] {\scriptsize 1/5} (nProc);

\path (nOff) edge[loop below,looseness=4, in=255, out=285] node[right,pos=0.2] {\scriptsize 1/0} (nOff);

\path (nIdle) edge[loop below,looseness=4, in=255, out=285] node[right,pos=0.2] {\scriptsize 1/1} (nIdle);
        \end{tikzpicture}
      \caption{Machine with $\StateIdle$ state.}
      \label{fig:isolines_idle}
    \end{subfigure}

    \vspace{1em}
    
    \begin{subfigure}{.7\textwidth}
    \centering
    \begin{tikzpicture}
        \begin{axis}[
            width=0.75\textwidth,
            height=0.25\textwidth,
            ybar, 
            label style={font=\small},
            tick label style={font=\scriptsize},
            xtick=data,
            xticklabel=\empty,
            ylabel={cost $\EnergyCost{\IdxInterval}$},
            xlabel={interval $\IdxInterval$ [-]},
            ymajorgrids,
            bar width=3pt,
            enlarge x limits={abs=0.75cm}, 
        ]
        \addplot table[x=idx, y=cost, col sep=comma] {other/isolines_c.txt};
        \end{axis}
    \end{tikzpicture}
    \caption{Example cost vector $\EnergyCostVector$ of TOU energy tariff.}\label{fig:isolines_cost}
    \end{subfigure}
    \caption{TEC isolines of a machine including and neglecting $\StateIdle$ state.}
    \label{fig:isolines}
\end{figure}
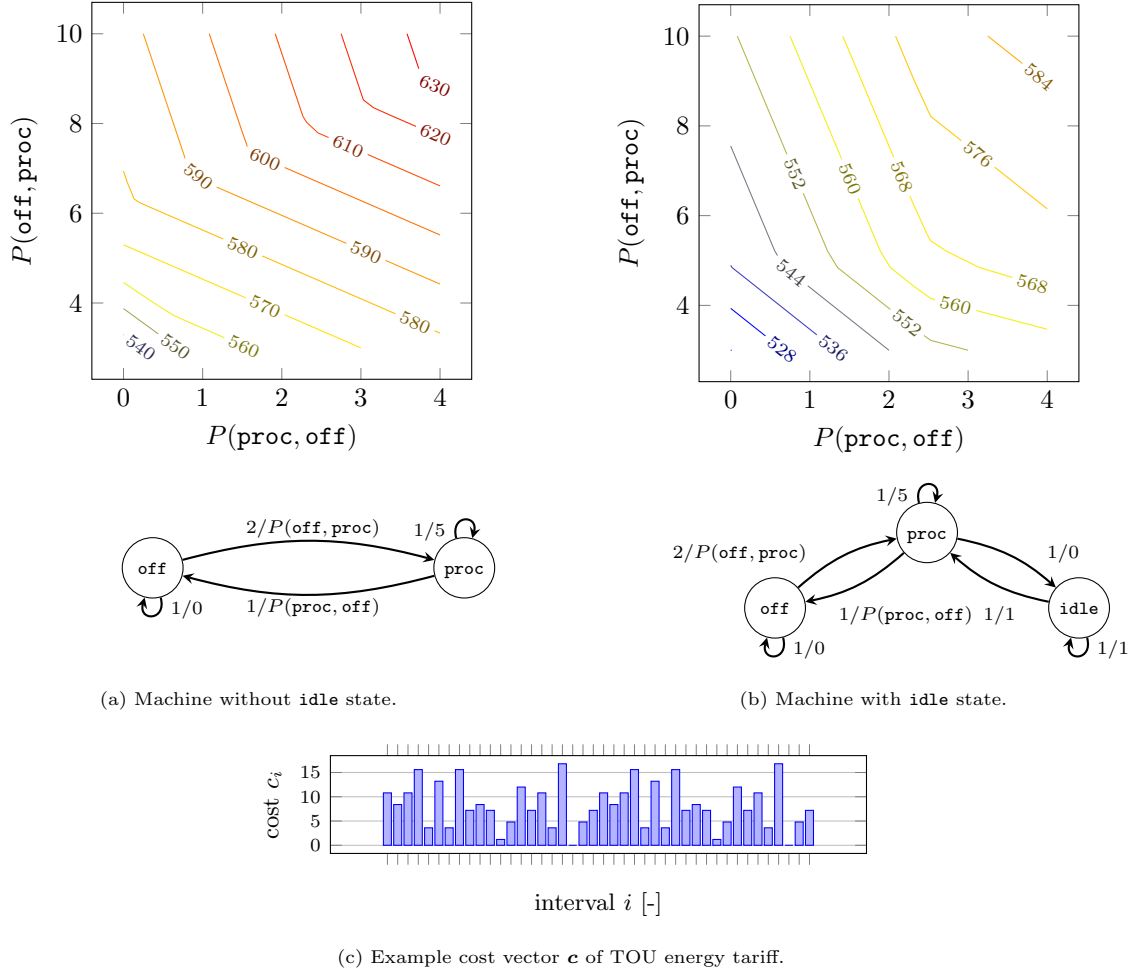
\end{Example}

The above-defined problem was introduced in \cite{2014:Shrouf} and is denoted in an extended Graham's notation as $\ProblemNotation$. 
The problem was shown to be strongly $\textsf{NP}$-hard in \cite{aghe2019toustates}.



\section{Related work}\label{sec:relwork}


The scheduling literature on energy-efficient scheduling and sustainable manufacturing concentrates primarily on the problems concerning Time-of-Use energy tariffs and the minimization of total energy costs~\citep{catanzaro2023jobsched_survey}.
The fundamental setting in these problems is that we are given a finite time horizon discretized into time intervals (e.g., 15 minutes), each associated with a cost for processing jobs inside the interval.
Jobs have processing time expressed in terms of an integer number of intervals.
The goal is to schedule all jobs on a single machine such that the total energy consumption is minimized, possibly subject to a constraint on the schedule makespan.
This basic scheduling problem with the time-of-use costs was coined by \cite{wan2010scheduling}.

Even the basic problem is quite practical and can be used to obtain some cost savings.
However, to match more complex use cases, many different variants were proposed.
In that respect, the most prominent research lines consider developing more complex machine environments, such as parallel identical machines~\citep{anghinolfi2021bi,tian2024single,gaggero2023exact}, unrelated machines~\citep{rego2022mathematical,9104021}, storage resources~\citep{ngueveu2022lower}, flow shops~\citep{ho2022exact,aghelinejad2020energy} or job shops~\citep{park2022energy}.

Other extensions of the problem consider different objective functions beyond TEC, mostly to express also the timelines of the schedule.
Among the most popular are bi-criteria and multi-objective approaches combining TEC with the makespan \citep{chen2021optimal,9104021} and total (weighted) tardiness \citep{rocholl2020bi, DING2021105088}.




Different approaches for bringing the scheduling models closer to reality consider more complex models of the job.
For example, considering non-uniform energy consumption~\citep{7807233}, speed~\citep{nattaf2017cumulative,fatih2018energy} (i.e., faster processing of the jobs in exchange for larger energy consumption), non-integer start times~\citep{CHEN2019900}, release times~\citep{WU2023136228} or scheduling with disjoint setup operations of energy-demanding machine reconfigurations~\citep{gnatowski2022scheduling}.
Typically, these were investigated in a single-machine setting first and later extended to more complex resource environments~\citep{wang2020multi}.

The complexity of particular problems also depends on the assumed properties of the energy cost vector~$\EnergyCostVector$.
For example, if $\EnergyCostVector$ takes a "pyramidal" shape (i.e., increasing and then decreasing), then the problem can become solvable in polynomial time~\citep{fang2016singlemach_tou}.
However, in realistic costs, this shape typically does not occur.
Further results of \cite{CHEN2019900} show that the problem complexity can depend on the number of "valleys" of the energy cost vector $\EnergyCostVector$.
The basic problem with makespan minimization, where energy consumption represents a hard constraint, does admit a 1.34-approximation algorithm~\citep{LI2024882}.

However, all the above extensions of the basic time-of-use scheduling problem disregard the fact that often the manufacturing resources are stateful machines~\citep{2018:Aghelinejad,2020a:Benedikt,2007:mouzon,2014:Shrouf}---meaning the machine has several operational modes, such as off, power saving, and processing mode.
The transitions between machine states are governed by the control logic and machine dynamics.
The operations on the jobs are performed only when the machine is in one of its processing states, while significant energy savings can be achieved by switching the machine to idle mode or even turning it off when the energy price is high.
Switching between machine operational modes adds another layer of complexity to the already \textsf{NP}-hard scheduling problem with time-of-use costs.
In this sense, the previous works that disregarded machine dynamics and operation modes can be seen as simplified problems where one assumes that the machine operates only in its processing state uninterrupted over the whole scheduling horizon.

The problem with simultaneous scheduling of jobs and switching operational states of the machine was first introduced by~\cite{2014:Shrouf}, denoted in three-field scheduling notation as $\ProblemNotation$.
The authors proposed an Integer Linear Programming (ILP) model for a single-machine environment and fixed order of the jobs.
Later, \cite{aghe2019toustates}~proved the problem without fixed order is hard, but under the fixed order of the jobs, it becomes polynomial-time solvable.
The restriction on the fixed order of jobs was lifted in \cite{2018:Aghelinejad}, where an ILP model was proposed to finally solve the original problem.
However, this approach was able to solve instances with up to 35 jobs to optimality.
Different variants of the problems were further studied.
A case of the problem with preemptive jobs was studied in \cite{2017Aghelinejad_preemption} or extensions of the problem into flow shop setting~\citep{AGHELINEJAD202011156,Aghelinejad_2020}.

Later, \cite{2020a:Benedikt}~realized that the performance of ILP and Constraint Programming (CP) models can be greatly enhanced by utilizing a pre-processing technique called SPACES that precomputes the optimal switching behavior of the machine between different intervals (since the costs differ) in advance and using these precomputed sequences inside the ILP and CP models.
This led to a new state-of-the-art performance, and the proposed algorithms were able to solve instances with 190 jobs that were up to optimality.
However, some of the problem structures were under-utilized and unexplored due to the limited flexibility of the mathematical programming models.
For example, their ILP model could not exploit the result of \cite{aghe2019toustates} that the problem is polynomially solvable under the fixed order of the jobs.
However, perhaps more importantly, the way in which the SPACES algorithm determines the optimal switching behavior between every pair of intervals and states suggests that, although not apparent at first sight, the precomputed switching creates so-called \textit{spaces} that can be used to form bins for a kind-of bin packing problem.
Thus, the problem inherently consists of a packing problem that can be typically solved very efficiently in practice.

In this paper, we further explore the properties of the $\ProblemNotation$ problem.
To exploit them inside an optimization algorithm fully, instead of formulating a new mathematical programming model, we design a custom branch-and-bound algorithm utilizing techniques of \cite{aghe2019toustates,2020a:Benedikt} and new observations made in this work that gave us enough flexibility to reach $100\times$ speedup over the former state-of-the-art algorithm~\citep{2020a:Benedikt}.
Furthermore, the proposed techniques are not only tied to this specific problem; thus, we believe they can be applied to similar problems considering time-of-use scheduling.

\section{Branch-and-Bound Algorithm for \( \ProblemNotation \)}\label{sec:bab}

For the scheduling problem $\ProblemNotation$ defined in \cref{sec:probstat}, we develop a fast, exact branch-and-bound algorithm called \AbbrBAB{}. 
Generally, a Branch and Bound is an exact algorithm for solving combinatorial problems.    
It recursively partitions the solution space into a tree using specific branching rules.
By employing a lower bound, parts of the solution space with an objective value that is worse than the currently best value can be pruned and excluded from the search.

Before explaining details of \AbbrBAB{}, we discuss important properties of the problem required for designing the algorithm, namely the optimal TEC computation for a fixed jobs sequence presented in \cref{sec:fixseq_tec}.
Further, we rely on the technique of pre-processing the optimal state switching of \cite{2020a:Benedikt}, with its details briefly summarized in \cref{sec:spaces}.


To design an efficient branch-and-bound algorithm, multiple parts of the algorithm are important, especially (i) an efficient branching rule, (ii) a tight lower bound that can be obtained in a reasonable time, and (iii) a good initialization upper bound with primal heuristics to prune the tree as soon as possible.
The rest of the section explains the individual components of \AbbrBAB{}.

\subsection{Optimal TEC Computation for a Fixed Jobs Sequence}\label{sec:fixseq_tec}

In~\cite{aghe2019toustates}, the problem of computing the optimal TEC with machine states for a fixed sequence of jobs was shown to be polynomial.
The authors solve this problem by finding the shortest path in a \DefTerm{job-interval} graph that models both the cost of scheduling the jobs and the cost of transiting between the machine states.

In this work, instead of modeling the transitions between the machine states explicitly, we improve the approach proposed in~\cite{aghe2019toustates} by incorporating only the optimal switching costs found by solving the optimal switching problem by a technique called \AbbrIwa{}~\citep{2020a:Benedikt}.
This way, we obtain a job-interval graph with a size independent of the size of the transition graph.
Thus, the optimal TEC can be found in a shorter time, which is critical for its usage inside a branch-and-bound algorithm, where it needs to be evaluated in every search node.
See \cref{sec:spaces} for more details about the optimal switching problem and \AbbrIwa{} technique.


Let \( \FixSeq: \IntInterval{1}{\NumJobs} \rightarrow \IntInterval{1}{\NumJobs} \) be a bijective function representing a \DefTerm{fixed jobs sequence} on the machine, i.e., \( \FixSeq[\FixSeqPos] \) maps position \( \FixSeqPos \) in the sequence to job index \( \IdxJob \). 
Given fixed jobs sequence \( \FixSeq \), the set of intervals in which job on position \( \FixSeqPos \in \IntInterval{1}{\NumJobs} \) can be scheduled, is denoted as
\begin{equation}
    \SetProcIntervals[\FixSeqPos] = \{ \Interval{\IdxInterval} \in \SetProcIntervals : \texttt{\FirstOn} + \sum_{\IdxAnother{\FixSeqPos} = 1}^{\FixSeqPos - 1} \ProcTime{\FixSeq[\IdxAnother{\FixSeqPos}]} \le \IdxInterval \le \texttt{\LastOn} - \sum_{\IdxAnother{\FixSeqPos} = \FixSeqPos}^{\NumJobs} \ProcTime{\FixSeq[\IdxAnother{\FixSeqPos}]} + 1\}\,.
\end{equation}
where \texttt{\FirstOn} and \texttt{\LastOn} are indices of $\Interval{\FirstOn}$ and $\Interval{\LastOn}$, respectively.
The definition of $\SetProcIntervals[\FixSeqPos]$ takes into account early/late \( \StateProc \) intervals (to start up and turn off the machine) and the time required for scheduling the jobs on positions before/after \( \FixSeqPos \).

To find the optimal TEC for fixed jobs sequence \( \FixSeq \), let us define a job-interval graph $G^{\text{ji}}$ by a triplet \( G^{\text{ji}}=(\SetVertTec, \SetEdgesTec, \WeightTec) \), where \( \SetVertTec \) is the set of \DefTerm{vertices}, \( \SetEdgesTec \) is the set of \DefTerm{edges} and \( \WeightTec: \SetEdgesTec \rightarrow \SetIntNonNeg \) are the \DefTerm{weights} of the edges.
The set of the vertices and edges of the job-interval graph are defined as follows:
\begin{align}
    \SetVertTec & = \{ \VertTecSource \} \cup \{ \VertTec{\FixSeqPos}{\IdxInterval} : \FixSeqPos \in \IntInterval{1}{\NumJobs}, \Interval{\IdxInterval} \in \SetProcIntervals[\FixSeqPos] \} \cup \{ \VertTecSink \}, \\
    \begin{split}
        \SetEdgesTec & = \{ (\VertTecSource, \VertTec{1}{\IdxInterval}) : \Interval{\IdxInterval} \in \SetProcIntervals[1] \}  \\
        & \phantom{{}={}} \cup \{ (\VertTec{\FixSeqPos}{\IdxInterval}, \VertTec{\FixSeqPos+1}{\IdxAnother{\IdxInterval}}) : \FixSeqPos \in \IntInterval{1}{\NumJobs - 1}, \Interval{\IdxInterval} \in \SetProcIntervals[\FixSeqPos], \Interval{\IdxAnother{\IdxInterval}} \in \SetProcIntervals[\FixSeqPos + 1], \IdxInterval + \ProcTime{\FixSeq[\FixSeqPos]} \le \IdxAnother{\IdxInterval}  \}, \\
        & \phantom{{}={}} \cup \{ (\VertTec{\NumJobs}{\IdxInterval}, \VertTecSink) : \Interval{\IdxInterval} \in \SetProcIntervals[\NumJobs]  \}\,.
    \end{split}
\end{align}
Informally, each edge \( (\VertTec{\FixSeqPos}{\IdxInterval}, \VertTec{\FixSeqPos+1}{\IdxAnother{\IdxInterval}}) \) represents a scheduling of job $\Job{\FixSeq[\FixSeqPos]}$ at the beginning of interval $\Interval{\IdxInterval}$ and job $\Job{\FixSeq[\FixSeqPos + 1]}$ at beginning of interval $\Interval{\IdxAnother{\IdxInterval}}$.
Vertices \( \VertTecSource \) and \( \VertTecSink \) correspond to the first and the last \( \StateOff \) state of the machine, respectively.

The weights of the edges take both the jobs processing costs and the optimal switching costs into account, which is in contrast to the original idea presented in~\cite{aghe2019toustates}, where the costs of transitions between machine states are considered.
Different edge types have different weights:
\begin{itemize}
    \item edges from the first \( \StateOff \) state to the first job:
    \begin{equation}
        \WeightTec[\VertTecSource, \VertTec{1}{\IdxInterval}] = \EnergyCost{1} \cdot \TransPower[\StateOff, \StateOff] + \OptCostTrans[1, \IdxInterval]\,;
    \end{equation}

    \item edges between two consecutive jobs:
    \begin{equation}
        \WeightTec[\VertTec{\FixSeqPos}{\IdxInterval}, \VertTec{\FixSeqPos+1}{\IdxAnother{\IdxInterval}}] = \JobCost{\IdxJob}{\FixSeq[\FixSeqPos]} + \OptCostTrans[\IdxInterval + \ProcTime{\FixSeq[\NumJobs]} - 1, \IdxAnother{\IdxInterval}]\,,
    \end{equation}
    where
    \begin{equation}
        \JobCost{\IdxJob}{\IdxInterval} = \sum_{\IdxAnother{\IdxInterval} = \IdxInterval}^{\IdxInterval + \ProcTime{\IdxJob} - 1} \EnergyCost{\IdxAnother{\IdxInterval}} \cdot \TransPower[\StateProc, \StateProc]
    \end{equation}
    is the cost of scheduling job \( \Job{\IdxJob} \) at the beginning of interval \( \Interval{\IdxInterval} \);

    \item edges from the last job to the last \( \StateOff \) state:
    \begin{equation}
        \WeightTec[\VertTec{\NumJobs}{\IdxInterval}, \VertTecSink] = \JobCost{\FixSeq[\NumJobs]}{\IdxInterval} + \OptCostTrans[\IdxInterval + \ProcTime{\FixSeq[\NumJobs]} - 1, \NumIntervals] + \EnergyCost{\NumIntervals} \cdot \TransPower[\StateOff, \StateOff]\,.
    \end{equation}
\end{itemize}

To obtain the optimal TEC for the fixed jobs sequence, it suffices to find the cost of the shortest path from \( \VertTecSource \) to \( \VertTecSink \). 
For a given sequence $\FixSeq$, we denote the optimal TEC by $\TotalEnergyCost{\FixSeq}$.
Notice that the job-interval graph is topologically ordered according to the lexicographic ordering of the labels \( (\FixSeqPos, \IdxInterval) \) of the vertices.
Therefore, the shortest path can be found by processing the vertices along this topological ordering, which can be done in \( \AsymCompUpper{\NumIntervals^2 \cdot \NumJobs } \) time by using, e.g., Bellman-Ford algorithm for DAGs that runs in $\AsymCompUpper{|\SetVertTec| + |\SetEdgesTec|}$ time. 
The complexity of the pre-processing for obtaining the optimal switching costs by SPACES is $\AsymCompUpper{\NumIntervals^2 \cdot \NumStates \cdot (\NumStates + \log \NumIntervals +\log \NumStates)}$ (for non-negative energy costs $\EnergyCostVector$, see \cref{sec:spaces} for more details).
Hence, when added up, this represents an improvement over the job-interval graph proposed in~\cite{aghe2019toustates}, where finding the shortest path takes \( \AsymCompUpper{\NumIntervals^2 \cdot \sum_{\Job{\IdxJob} \in \SetJobs} \ProcTime{\IdxJob}} \), since here we avoid the pseudo-polynomial factor in terms of the processing times.
Also, note that since the graph is acyclic, we can still handle instances with negative energy costs as these are encapsulated inside weights $\WeightTec$ of the job-interval graph.


\begin{figure}
    \centering
    \input{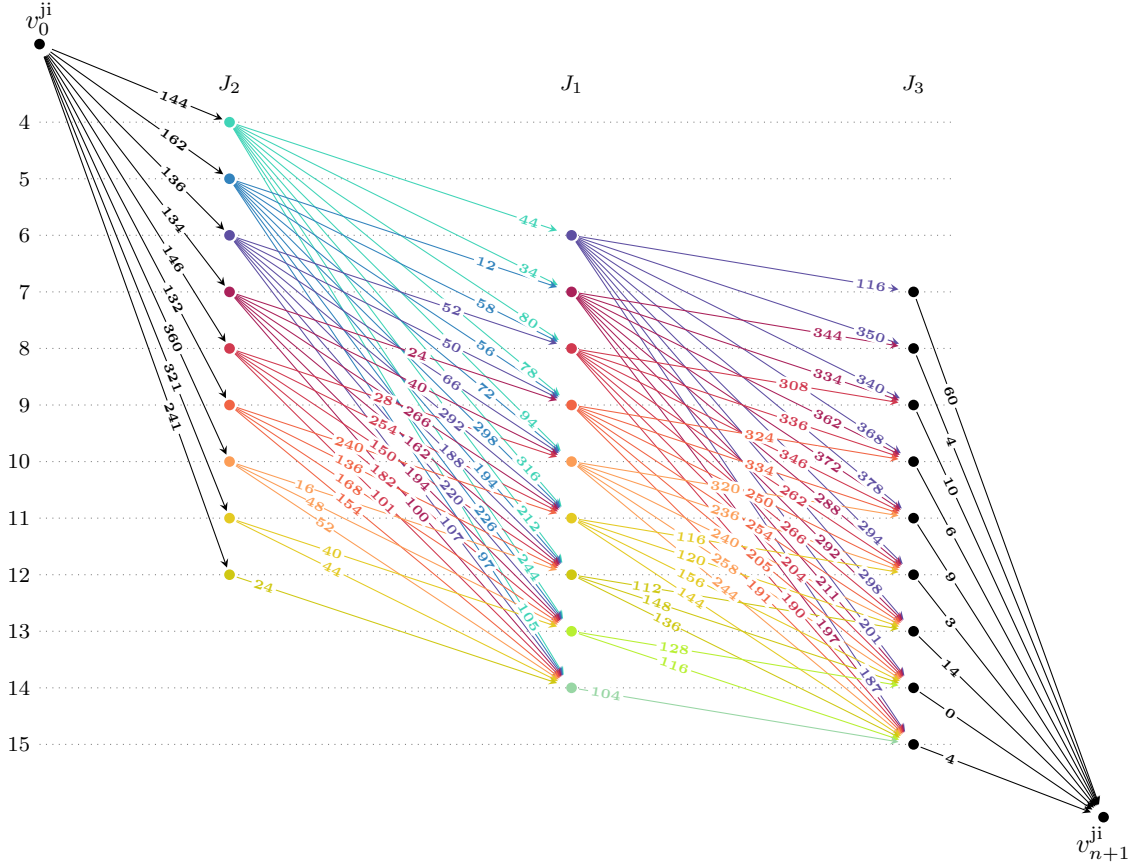}
    \caption{Example job-interval graph.}
    \label{fig:interval-jobs-graph}
\end{figure}

\subsection{The branching strategy and the lower bound}\label{sec:bab-lower-bound}


\AbbrBAB{} progressively constructs the jobs sequence \( \FixSeq \).
In each node of the search tree, \AbbrBAB{} keeps the current \DefTerm{partial job sequence} \( \FixSeqPartial : \IntInterval{1}{|\SetJobsPartial| } \rightarrow \SetJobsPartial \) for some jobs subset \( \SetJobsPartial \subseteq \SetJobs \).
The algorithm branches on the remaining un-fixed jobs \( \SetJobs \setminus \SetJobsPartial \), i.e., a new branch is created by appending one of the remaining jobs to the end of \( \FixSeqPartial \).
To break the symmetries arising from the jobs having equal processing times, only one branch per unique remaining processing time is created in each node.
To understand the branching, see an example of a branching subtree in \cref{fig:bab-basic} corresponding to the example in \cref{ex:example1} with three jobs $\Job{1}, \Job{2}, \Job{3}$ with processing times $\{1,2,4\}$, respectively.

The lower bound for each node of the branching tree is computed using the method for a fixed sequence utilizing the job-interval graph described in~\cref{sec:fixseq_tec}.
Let \( \FixSeqPartial : \IntInterval{1}{|\SetJobsPartial| } \rightarrow \SetJobsPartial \) be the partial job sequence in a search tree node for some jobs subset \( \SetJobsPartial \subseteq \SetJobs \).
To compute the lower bound, a new set of jobs \( \SetJobsLb \) is created by splitting the remaining un-fixed jobs into unit processing times jobs, i.e.,
\begin{equation} \label{eq:relaxed-jobs-lb}
    \SetJobsLb = \SetJobsPartial \cup \{ \JobPart{\IdxJob} : \IdxJob \in \IntInterval{1}{\sum_{\Job{\IdxAnother{\IdxJob}} \in \SetJobs \setminus \SetJobsPartial} \ProcTime{\IdxAnother{\IdxJob}}}  \}\,,
\end{equation}
where  \( \JobPart{\IdxJob} \) is a new job with unit processing time.
From \( \SetJobsLb \), a jobs sequence \( \FixSeqPartialLb \) is defined as follows
\begin{equation} \label{eq:fix-sequence-lb}
    \FixSeqPartialLb[\FixSeqPos]=
    \begin{dcases}
        \FixSeqPartial[\FixSeqPos] & \FixSeqPos \le |\SetJobsPartial|, \\
        \JobPart{\FixSeqPos - |\SetJobsPartial|} & |\SetJobsPartial| < \FixSeqPos  \le |\SetJobsLb|.
    \end{dcases}
\end{equation}
The lower bound $lb=\TotalEnergyCost{\FixSeqPartialLb}$ in the search tree node is then obtained by running the method from~\cref{sec:fixseq_tec} 
on jobs sequence \( \FixSeqPartialLb \).
See an example of the branching tree with computed lower bounds values in \cref{fig:bab-basic}.
The tree is traversed in a depth-first search, starting with the leftmost child.
At the leaves of the search tree, we obtain a feasible solution with the objective value equal to $ub$, representing an upper bound on the optimal objective value that is used to prune subtrees with the lower bound $lb\geq ub$.
The following proposition formally states the lower bound.


\begin{proposition}
Let \( \FixSeqPartial: \IntInterval{1}{|\SetJobsPartial| } \rightarrow \SetJobsPartial  \) be a partial sequence of jobs \( \SetJobsPartial \subseteq \SetJobs \), and \(\FixSeqPartialLb\) be the relaxed sequence defined by \eqref{eq:fix-sequence-lb}. Then, for any permutation $\FixSeq: \IntInterval{1}{|\SetJobs|} \rightarrow \SetJobs$ satisfying \(\FixSeq[\FixSeqPos] = \FixSeqPartial[\FixSeqPos], \forall \FixSeqPos \in \IntInterval{1}{|\SetJobsPartial|} \) it holds that \( \TotalEnergyCost{\FixSeq} \geq \TotalEnergyCost{\FixSeqPartialLb} \).
\end{proposition}

Later in~\cref{sec:bab-lb-gcd}, we will prove a slightly stronger version of this proposition.

\begin{figure}
    \centering
    \begin{tikzpicture}
\coordinate (TreeLayerHeight) at (0cm,1.75cm);
\coordinate (TreeLeafDistance) at (2.5cm,0cm);
\newcommand{\crossNode}[1]{\node[gray] at (#1) {\xmark};}

\tikzset{BABNode/.style = {circle, fill=white, draw=black, inner sep=2pt, align=center, thick},
    labelNode/.style = {circle, fill=white, inner sep=1pt, outer sep=3pt, midway,font=\small,text=blue},
    lbNode/.style = {anchor=south east, xshift=-5pt, font=\scriptsize, fill=white, rectangle, inner sep=1pt, outer sep=1pt},
    ubNode/.style = {anchor=north, yshift=-5pt, font=\scriptsize},
    pruned/.style = {draw=gray!50!white, text=gray!50!white}}

\node[BABNode] (n1-1) at ($2.5*(TreeLeafDistance)-0.0*(TreeLayerHeight)$) {};
\node[BABNode] (n2-1) at ($0.5*(TreeLeafDistance)-1.0*(TreeLayerHeight)$) {};
\node[BABNode] (n2-2) at ($2.5*(TreeLeafDistance)-1.0*(TreeLayerHeight)$) {};
\node[BABNode] (n2-3) at ($4.5*(TreeLeafDistance)-1.0*(TreeLayerHeight)$) {};
\node[BABNode] (n3-1) at ($0.0*(TreeLeafDistance)-2.0*(TreeLayerHeight)$) {};
\node[BABNode] (n3-2) at ($1.0*(TreeLeafDistance)-2.0*(TreeLayerHeight)$) {};
\node[BABNode] (n3-3) at ($2.0*(TreeLeafDistance)-2.0*(TreeLayerHeight)$) {};
\node[BABNode] (n3-4) at ($3.0*(TreeLeafDistance)-2.0*(TreeLayerHeight)$) {};
\node[BABNode, pruned] (n3-5) at ($4.0*(TreeLeafDistance)-2.0*(TreeLayerHeight)$) {}; 
\node[BABNode, pruned] (n3-6) at ($5.0*(TreeLeafDistance)-2.0*(TreeLayerHeight)$) {}; 
\node[BABNode] (n4-1) at ($0.0*(TreeLeafDistance)-3.0*(TreeLayerHeight)$) {};
\node[BABNode, pruned] (n4-2) at ($1.0*(TreeLeafDistance)-3.0*(TreeLayerHeight)$) {}; 
\node[BABNode] (n4-3) at ($2.0*(TreeLeafDistance)-3.0*(TreeLayerHeight)$) {};
\node[BABNode, pruned] (n4-4) at ($3.0*(TreeLeafDistance)-3.0*(TreeLayerHeight)$) {}; 
\node[BABNode, pruned] (n4-5) at ($4.0*(TreeLeafDistance)-3.0*(TreeLayerHeight)$) {}; 
\node[BABNode, pruned] (n4-6) at ($5.0*(TreeLeafDistance)-3.0*(TreeLayerHeight)$) {}; 

\path (n1-1) edge node[labelNode] {1} (n2-1) 
             edge node[labelNode] {2} (n2-2)
             edge node[labelNode] {4} (n2-3);
             
\path (n2-1) edge node[labelNode] {2} (n3-1) 
             edge node[labelNode] {4} (n3-2);
             
\path (n2-2) edge node[labelNode] {1} (n3-3) 
             edge node[labelNode] {4} (n3-4);
            
\path (n2-3) edge[pruned] node[labelNode] {1} (n3-5) 
             edge[pruned] node[labelNode] {2} (n3-6); 

             
\path (n3-1) edge node[labelNode] {4} (n4-1);
\path (n3-2) edge[pruned] node[labelNode] {2} (n4-2);
\path (n3-3) edge node[labelNode] {4} (n4-3);
\path (n3-4) edge[pruned] node[labelNode] {1} (n4-4);
\path (n3-5) edge[pruned] node[labelNode] {2} (n4-5);
\path (n3-6) edge[pruned] node[labelNode] {1} (n4-6);

\node[lbNode] at (n1-1) {$lb=339$};
\node[lbNode] at (n2-1) {$lb=339$};
\node[lbNode] at (n3-1) {$lb=339$};
\node[lbNode] at (n4-1) {$lb=353$};
\node[ubNode] at (n4-1) {\uline{$ub=353$} $(\Job{1},\Job{2},\Job{3})$};

\node[lbNode] at (n3-2) {$lb=364$};

\node[lbNode] at (n2-2) {$lb=339$};
\node[lbNode] at (n3-3) {$lb=339$};
\node[lbNode] at (n4-3) {$lb=342$};
\node[ubNode] at (n4-3) {\uline{$ub=342$} $(\Job{2},\Job{1},\Job{3})$};

\node[lbNode] at (n3-4) {$lb=342$};

\node[lbNode] at (n2-3) {$lb=360$};

\end{tikzpicture}
    \caption{Branch-and-bound tree with lower bound $\FixSeqPartialLb$. The numbers on the edges correspond to the processing time of the job selected during the branching.}
    \label{fig:bab-basic}
\end{figure}
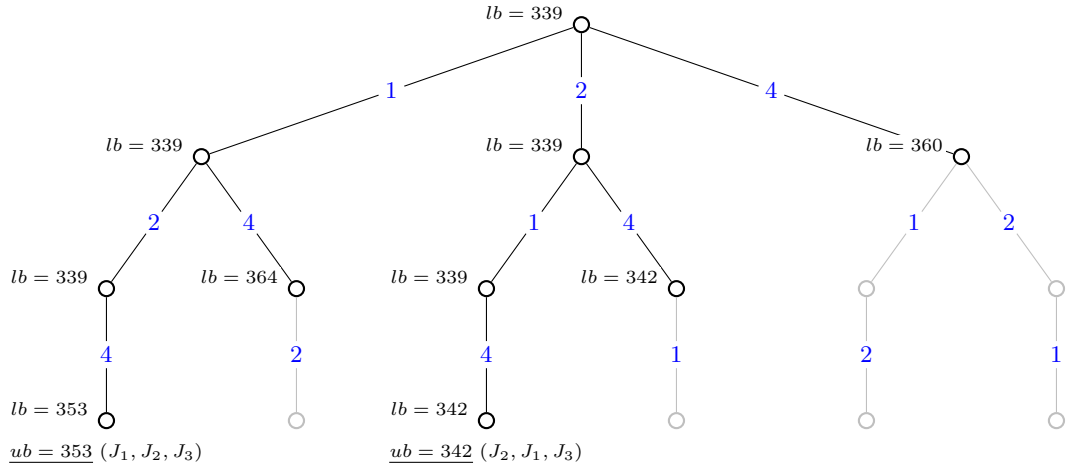

\paragraph{Implementation remarks}
 The computation of the lower bound between the nodes can be performed efficiently using only a two-dimensional array of the costs.
The rows of the array are called \DefTerm{levels}, and the array contains \( \sum_{\Job{\IdxJob} \in \SetJobs} \ProcTime{\IdxJob} \) levels.
The number of columns corresponds to \( |\SetProcIntervals| \).
Each element of the array can be seen as a corresponding vertex from the set \( \SetVertTec \) of the job-interval graph, where each job has a processing time of 1.

\begin{sloppypar}
When a job \( \Job{\IdxJob} \in \SetJobs \setminus \SetJobsPartial \) is appended to the end of the \( \FixSeqPartial \) in a search tree node, the levels \hbox{\( \IntInterval{\sum_{\Job{\IdxAnother{\IdxJob}} \in \SetJobsPartial} \ProcTime{\IdxAnother{\IdxJob}}}{\ProcTime{\IdxJob} + \sum_{\Job{\IdxAnother{\IdxJob}} \in \SetJobsPartial} \ProcTime{\IdxAnother{\IdxJob}}}\)} are \DefTerm{joined} together, i.e., as if the edges in the job-interval graph between these levels were removed.
On the other hand, on backtracking from a node, the levels are \DefTerm{split}, i.e., as if the edges in the job-interval graph between these levels are included again.
For the better utilization of the data structures, i.e., no copies of the memory-intensive objects, the \AbbrBAB{} is implemented as a depth-first search.
\end{sloppypar}

\begin{figure}
    \centering
    \begin{tikzpicture}
\coordinate (TreeLayerHeight) at (0cm,1.75cm);
\coordinate (TreeLeafDistance) at (2.5cm,0cm);
\newcommand{\crossNode}[1]{\node[gray] at (#1) {\xmark};}
\newcommand{\nodeLblGcdLb}[3]{\node[lbNode] (#1D) at (#1) {\begin{tabular}{rl}gcd & \;{=} #2\\  $lb$ & \;{=} #3\end{tabular}}; }
\newcommand{\nodeLblGcdLbR}[3]{\node[lbNodeR] at (#1) {\begin{tabular}{rl}gcd & \;{=} #2\\  $lb$ & \;{=} #3\end{tabular}}; }

\tikzset{BABNode/.style = {circle, fill=white, draw=black, inner sep=2pt, align=center, thick},
    labelNode/.style = {circle, fill=white, inner sep=1pt, outer sep=3pt, midway,font=\small,text=blue},
    lbNode/.style = {anchor=south east, xshift=-5pt, font=\scriptsize, fill=white, rectangle, inner sep=0pt, outer sep=0pt, align=left},
    lbNodeR/.style = {anchor=south west, xshift=3pt, font=\scriptsize, fill=white, rectangle, inner sep=0pt, outer sep=0pt, align=left},
    ubNode/.style = {anchor=north, yshift=-5pt, font=\scriptsize},
    pruned/.style = {draw=gray!50!white, text=gray!50!white}}

\node[BABNode] (n1-1) at ($2.5*(TreeLeafDistance)-0.0*(TreeLayerHeight)$) {};
\node[BABNode] (n2-1) at ($0.5*(TreeLeafDistance)-1.0*(TreeLayerHeight)$) {};
\node[BABNode] (n2-2) at ($2.5*(TreeLeafDistance)-1.0*(TreeLayerHeight)$) {};
\node[BABNode] (n2-3) at ($4.5*(TreeLeafDistance)-1.0*(TreeLayerHeight)$) {};
\node[BABNode] (n3-1) at ($0.0*(TreeLeafDistance)-2.0*(TreeLayerHeight)$) {};
\node[BABNode, pruned] (n3-2) at ($1.0*(TreeLeafDistance)-2.0*(TreeLayerHeight)$) {};
\node[BABNode] (n3-3) at ($2.0*(TreeLeafDistance)-2.0*(TreeLayerHeight)$) {};
\node[BABNode] (n3-4) at ($3.0*(TreeLeafDistance)-2.0*(TreeLayerHeight)$) {};
\node[BABNode, pruned] (n3-5) at ($4.0*(TreeLeafDistance)-2.0*(TreeLayerHeight)$) {}; 
\node[BABNode, pruned] (n3-6) at ($5.0*(TreeLeafDistance)-2.0*(TreeLayerHeight)$) {}; 
\node[BABNode] (n4-1) at ($0.0*(TreeLeafDistance)-3.0*(TreeLayerHeight)$) {};
\node[BABNode, pruned] (n4-2) at ($1.0*(TreeLeafDistance)-3.0*(TreeLayerHeight)$) {}; 
\node[BABNode] (n4-3) at ($2.0*(TreeLeafDistance)-3.0*(TreeLayerHeight)$) {};
\node[BABNode, pruned] (n4-4) at ($3.0*(TreeLeafDistance)-3.0*(TreeLayerHeight)$) {}; 
\node[BABNode, pruned] (n4-5) at ($4.0*(TreeLeafDistance)-3.0*(TreeLayerHeight)$) {}; 
\node[BABNode, pruned] (n4-6) at ($5.0*(TreeLeafDistance)-3.0*(TreeLayerHeight)$) {}; 

\path (n1-1) edge node[labelNode] {1} (n2-1) 
             edge node[labelNode] {2} (n2-2)
             edge node[labelNode] {4} (n2-3);
             
\path (n2-1) edge node[labelNode] {2} (n3-1) 
             edge[pruned] node[labelNode] {4} (n3-2);
             
\path (n2-2) edge node[labelNode] {1} (n3-3) 
             edge node[labelNode] {4} (n3-4);
            
\path (n2-3) edge[pruned] node[labelNode] {1} (n3-5) 
             edge[pruned] node[labelNode] {2} (n3-6); 

             
\path (n3-1) edge node[labelNode] {4} (n4-1);
\path (n3-2) edge[pruned] node[labelNode] {2} (n4-2);
\path (n3-3) edge node[labelNode] {4} (n4-3);
\path (n3-4) edge[pruned] node[labelNode] {1} (n4-4);
\path (n3-5) edge[pruned] node[labelNode] {2} (n4-5);
\path (n3-6) edge[pruned] node[labelNode] {1} (n4-6);

{
\setlength{\tabcolsep}{0pt}
\renewcommand{\arraystretch}{0.75}

\nodeLblGcdLb{n1-1}{1}{339}
\nodeLblGcdLb{n2-1}{2}{353}
\nodeLblGcdLb{n3-1}{4}{353}
\nodeLblGcdLb{n4-1}{$\emptyset$}{353}
\node[ubNode] at (n4-1) {\uline{$ub=353$} $(\Job{1},\Job{2},\Job{3})$};

\nodeLblGcdLb{n2-2}{1}{339}
\nodeLblGcdLb{n3-3}{4}{342}
\nodeLblGcdLb{n4-3}{$\emptyset$}{342}
\node[ubNode] at (n4-3) {\uline{$ub=342$} $(\Job{2},\Job{1},\Job{3})$};

\nodeLblGcdLbR{n3-4}{1}{342}
\nodeLblGcdLbR{n2-3}{1}{360}

\draw[very thick, Clr2] ($(n2-1.south east) + (0.2,-0.2)$) rectangle ($(n2-1D.north west) + (-0.2,0.2)$);

\draw[very thick, Clr2] ($(n3-3.south east) + (0.2,-0.2)$) rectangle ($(n3-3D.north west) + (-0.2,0.2)$);
}

\end{tikzpicture}
    \caption{Branch-and-bound with lower bound $\widehat{\pi}^{\text{lb}}_{\gcd}$ improved with $\gcd$ of the unfixed jobs. Bounds that were improved with respect to \cref{fig:bab-basic} are highlighted in a rectangle.}
    \label{fig:bab-gcd}
\end{figure}
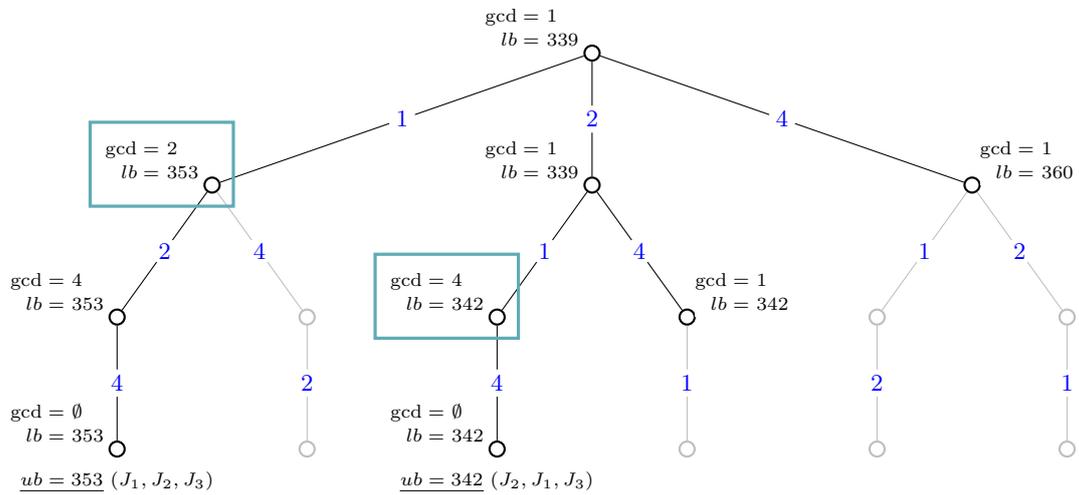

\subsection{Tightening the Lower Bound for Non-coprime Processing Times} \label{sec:bab-lb-gcd}
The lower bound introduced in Section~\ref{sec:bab-lower-bound} relies on the relaxation of the non-preemptive jobs into a set of jobs with unit processing times.
Since these jobs are interchangeable, they can be scheduled optimally using the job-interval graph under an arbitrary permutation.
The idea for improving this lower bound builds on the observation that particular sets of un-fixed jobs $\SetJobs \setminus \SetJobsPartial$ can be relaxed into jobs with processing times greater than one, thus improving the relaxation. 
We say that a set of positive integers is \DefTerm{non-coprime} if their greatest common divisor is greater than 1.
In this section, we explain how a tighter lower bound for jobs having non-coprime processing times can be obtained.

The idea of obtaining a tighter lower bound is similar to the one of $\TotalEnergyCost{\FixSeqPartialLb}$.
Let \( \SetJobs \setminus \SetJobsPartial \) be a set of un-fixed jobs in some node of the search tree.
Further, let 
$$\gcd(\SetJobs \setminus \SetJobsPartial) = \gcd(\{p_{j^\prime} :  J_{j^\prime} \in \SetJobs \setminus \SetJobsPartial \})$$
be the greatest common divisor of processing times of all un-fixed jobs.
We define a set of jobs 
\begin{equation} \label{eq:relaxed-jobs-gcd-lb}
    \mathcal{J}^\text{lb}_{\text{gcd}} = \SetJobsPartial \cup \{ \JobGCD{\IdxJob} : \IdxJob \in \IntInterval{1}{\gcd(\SetJobs \setminus \SetJobsPartial)^{-1}\cdot\sum_{\Job{\IdxAnother{\IdxJob}} \in \SetJobs \setminus \SetJobsPartial} \ProcTime{\IdxAnother{\IdxJob}}}  \}\,,
\end{equation}
where  \( \JobGCD{\IdxJob} \) is a new job with processing time equal to $\gcd(\SetJobs \setminus \SetJobsPartial)$.
Finally, the sequence \( \widehat{\pi}^{\text{lb}}_{\gcd} \) of jobs \(  \mathcal{J}^\text{lb}_{\text{gcd}} \) is defined similarly as \eqref{eq:fix-sequence-lb}, that is
\begin{equation} \label{eq:fix-sequence-gcd-lb}
    \widehat{\pi}^{\text{lb}}_{\gcd}=
    \begin{dcases}
        \FixSeqPartial[\FixSeqPos] & \FixSeqPos \le |\SetJobsPartial|, \\
        \JobGCD{\FixSeqPos - |\SetJobsPartial|} & |\SetJobsPartial| < \FixSeqPos  \le |\mathcal{J}^\text{lb}_{\text{gcd}}|.
    \end{dcases}
\end{equation}
Finally, a lower bound on TEC with respect to the partial solution $\SetJobsPartial$ is obtained by solving the problem with the set of jobs $\mathcal{J}^\text{lb}_{\text{gcd}}$ under sequence $\widehat{\pi}^{\text{lb}}_{\gcd}$ on job-interval graph by a procedure given in~\cref{sec:fixseq_tec}.

\begin{proposition}
Let \( \FixSeqPartial: \IntInterval{1}{|\SetJobsPartial| } \rightarrow \SetJobsPartial  \) be a partial sequence of jobs \( \SetJobsPartial \subseteq \SetJobs \), and $\FixSeqPartialLb$,  $\widehat{\pi}^{\text{lb}}_{\gcd}$ be the relaxed sequences defined by \eqref{eq:fix-sequence-lb} and \eqref{eq:fix-sequence-gcd-lb}. 
Then, for any permutation $\FixSeq: \IntInterval{1}{|\SetJobs|} \rightarrow \SetJobs$ satisfying \(\FixSeq[\FixSeqPos] = \FixSeqPartial[\FixSeqPos], \forall \FixSeqPos \in \IntInterval{1}{|\SetJobsPartial|} \) it holds that $\TotalEnergyCost{\FixSeq} \overset{(a)}{\geq} \TotalEnergyCost{\widehat{\pi}^{\text{lb}}_{\gcd}} \overset{(b)}{\geq} \TotalEnergyCost{\FixSeqPartialLb}$ with inequality~(b) being strict for some inputs.
\end{proposition}

\begin{proof}
The inequality $(a)$ follows from the observation that $\mathcal{J}^\text{lb}_{\text{gcd}}$ represents the problem that relaxes on the condition of non-preemptive jobs.
To show $\TotalEnergyCost{\widehat{\pi}^{\text{lb}}_{\gcd}}$ is indeed a lower bound on $\TotalEnergyCost{\FixSeq}$, one needs to ensure that sequence $\widehat{\pi}^{\text{lb}}_{\gcd}$ is the optimal solution of the relaxed problem, that is whether
$$\widehat{\pi}^{\text{lb}}_{\gcd} \in \arg \min_{\pi^\prime\in\Pi(  \mathcal{J}^\text{lb}_{\text{gcd}})} \TotalEnergyCost{\pi^\prime},$$
where $\Pi(\mathcal{J}^\text{lb}_{\text{gcd}})$ is the set of all permutations of jobs $\mathcal{J}^\text{lb}_{\text{gcd}}$. Focusing on the unfixed part of the jobs $\mathcal{J}^\text{lb}_{\text{gcd}}\setminus \SetJobsPartial$, one can see that all jobs $\JobGCD{\IdxJob}$ are identical with processing time equal to $\gcd(\SetJobs \setminus \SetJobsPartial)$.
Thus, all unfixed jobs are completely interchangeable. 
Therefore, we can employ the method from \cref{sec:fixseq_tec} applied to a job-interval graph with a fixed sequence to obtain the optimal solution of the relaxed problem.

Similarly, for inequality~$(b)$, sequence $\FixSeqPartialLb$ can be seen as an optimal solution of a relaxed problem with $\mathcal{J}^\text{lb}_{\text{gcd}}$ into unit-processing time jobs $\JobPart{\IdxJob}$ and the same arguments as above apply.
To see why the inequality~$(b)$ is strict for some instances (namely the ones with short sequences of intervals with low costs) when $\gcd > 1$, see, e.g., example branching trees in~\cref{fig:bab-basic} and \cref{fig:bab-gcd}.
An illustrative example would consider an instance where $\gcd=2$, with one interval $\Interval{\IdxInterval}$ having cost significantly smaller than its neighborhood intervals, i.e., $\min \{c_{i-1}, c_{i+1}\} \gg c_i$.
The unit processing time lower bound $\FixSeqPartialLb$ would exploit cheap interval $\Interval{\IdxInterval}$, whereas $\widehat{\pi}^{\text{lb}}_{\gcd}$ avoids it, leading to a stronger bound.
    


\end{proof}



\subsection{Bin-packing Primal Heuristic}\label{sec:binpack_primal_heur}


In order to find a good upper bound $ub$ as soon as possible, so-called \DefTerm{primal heuristics} are often used during the search. 
These heuristics try to reconstruct a feasible solution early in the search tree (possibly in each node of the tree), hoping to further prune the branches that do not contain any better solution.
In \AbbrBAB{}, we introduce a similar idea based on the observation that the solution to the relaxed problem can be, in some cases, used to recover a solution for the original problem that does not preempt the jobs.
In such cases, the cost of the recovered solution matches the lower bound value, efficiently pruning the search subtree.

In each node of the search tree, we first compute a lower bound for either jobs sequence given by \( \FixSeqPartialLb \) or $\widehat{\pi}^{\text{lb}}_{\gcd}$ using the method described in~\cref{sec:bab-lower-bound}.
From the solution of the relaxed instance, we extract the maximal sequences of consecutive intervals during which the machine operates in the \( \StateProc \) state.
Let us denote these sequences by \( (\ProcSequence[1], \ProcSequence[2], \dots, \ProcSequence[k],\ldots, \ProcSequence[n(\FixSeqPartialLb)]) \), where \( \ProcSequence[\IdxProcSequence] \subseteq \SetProcIntervals \) and \( n(\FixSeqPartialLb) \) 
is the number of these maximal sequences in the optimal solution of the relaxed problem for \( \FixSeqPartialLb \) (or $\widehat{\pi}^{\text{lb}}_{\gcd}$). 
Now, we may formulate the following \( \ProblemBinPack \) bin-packing problem:

\begin{sloppypar}
\paragraph{\( \ProblemBinPack \) problem} Given the maximal processing intervals sequences \( ( \ProcSequence[1], \dots, \ProcSequence[k],\ldots, \ProcSequence[n(\FixSeqPartialLb)] ) \) and given the jobs  \(\SetJobs \), find a partition of \( \SetJobs \) into disjoint sets \( \SetJobs[1] \), \( \ldots, \SetJobs[k]\), \dots, \(\SetJobs[n(\FixSeqPartialLb)] \) such that
\begin{equation}
    \forall \IdxProcSequence \in \IntInterval{1}{n(\FixSeqPartialLb)}: \sum_{\Job{\IdxJob} \in \SetJobs[\IdxProcSequence]} \ProcTime{\IdxJob} \leq |\ProcSequence[\IdxProcSequence]|\,,
\end{equation}
where $|\ProcSequence[\IdxProcSequence]|$ is the length of $\IdxProcSequence$-th sequence of consecutive processing intervals.
\end{sloppypar}

The \( \ProblemBinPack \) problem resembles the decision version of a bin-packing problem with unequal bin sizes, which is known to be $\textsf{NP}$-complete.
If the \( \ProblemBinPack \) problem instance is solved in a node of the branch and bound tree and is feasible, we obtain a feasible solution to \( \ProblemNotation \) with the objective value $ub$.
The $ub$ value is the total energy cost of that solution and is equal to the lower bound \( \TotalEnergyCost{\FixSeqPartialLb} \).
Therefore, the whole sub-tree rooted in the current node can be pruned since the lower bound may only increase within the sub-tree.
On the other hand, if the \( \ProblemBinPack \) problem is not solved, the \AbbrBAB{} continues branching.

\begin{Example}{3}
    Consider the node highlighted by a green rectangle in the branch and bound tree in \cref{fig:bab-packing}.
    When computing a lower bound for the unfixed job with original processing times \{1,4\} relaxed into five unit-sized jobs $\JobPart{\IdxJob}$ since $\gcd{\{1,4\}}=1$. 
    The shortest path on the job-interval graph from \cref{sec:fixseq_tec} leads to a solution considering the optimal switching behavior
    \begin{equation*}
        \begin{split}
            \OptPathTrans[1,20] = &\left((\StateOff, \StateOff), (\StateOff, \StateOff), (\StateOff, \StateOff), (\StateOff,\StateProc), (\StateProc, \StateProc),(\StateProc, \StateProc), (\StateProc,\StateOff),\right.\\
            &\left.(\StateOff,\StateProc),(\StateProc, \StateProc), (\StateProc,\StateProc), (\StateProc,\StateIdle), (\StateIdle,\StateProc), (\StateProc,\StateOff), (\StateOff,\StateOff)\right)
        \end{split}
    \end{equation*} 
    with the cost of 339.
    This induces an instance of $\ProblemBinPack$ problem with $( \ProcSequence[1], \ProcSequence[2], \ProcSequence[3])$ having $|\ProcSequence[1]|=3$, $|\ProcSequence[2]|=3$ and $|\ProcSequence[3]|=1$.
    Such an instance is infeasible. Thus, branching continues with the sequence $(\Job{2},\Job{1})$, where $\gcd$ of unfixed jobs is 4.
    Computing the lower bound at this node, i.e., the optimal solution for sequence $\widehat{\pi}^{\text{lb}}_{\gcd}$, yields a relaxed solution with blocks $|\ProcSequence[1]|=2$ and $|\ProcSequence[2]|=5$ (identical to the one shown in \cref{fig:example-schedule}), which can be used to pack all jobs with processing times $\{1, 2, 4\}$, producing a feasible solution with $c^{TEC}$ equal to $ub=342$, matching a lower bound value $\TotalEnergyCost{\widehat{\pi}^{\text{lb}}_{\gcd}}$.   
\end{Example}

\paragraph{Implementation remarks} The packing problem \( \ProblemBinPack \) can be tackled by any standard approach, including the heuristics 
or the exact algorithms. 
In our case, we solve \( \ProblemBinPack \) using IBM CP Optimizer with a single \( \CmdCpPack \) global constraint.
Since the bin-packing problem is $\textsf{NP}$-complete in general, the time limit has to be set for the solver to avoid stalling, although in practice, such problems are solved quickly.



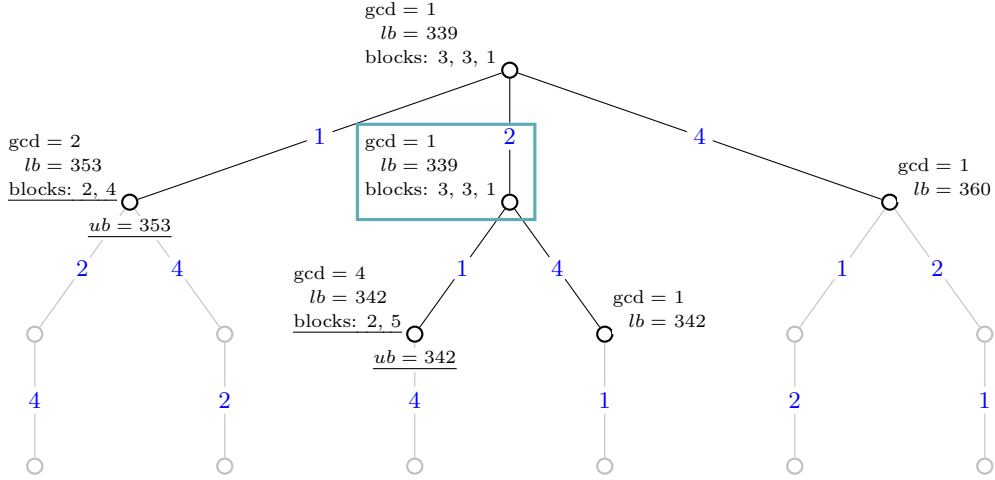
\begin{figure}
    \centering
    \begin{tikzpicture}
\coordinate (TreeLayerHeight) at (0cm,1.75cm);
\coordinate (TreeLeafDistance) at (2.5cm,0cm);
\newcommand{\crossNode}[1]{\node[gray] at (#1) {\xmark};}
\newcommand{\nodeLblGcdLbPack}[4]{\node[lbNode] at (#1) {\begin{tabular}{rl}gcd & \;{=} #2\\  $lb$ & \;{=} #3\\\multicolumn{2}{c}{blocks: #4}\end{tabular}}; }

\newcommand{\nodeLblGcdLbPackU}[4]{\node[lbNode] at (#1) {\begin{tabular}{rl}gcd & \;{=} #2\\  $lb$ & \;{=} #3\\\multicolumn{2}{c}{\uline{blocks: #4}}\end{tabular}}; }
\newcommand{\nodeLblGcdLbPackR}[4]{\node[lbNodeR] at (#1) {\begin{tabular}{rl}gcd & \;{=} #2\\  $lb$ & \;{=} #3\\\multicolumn{2}{c}{blocks: #4}\end{tabular}}; }

\newcommand{\nodeLblGcdLbR}[3]{\node[lbNodeR] at (#1) {\begin{tabular}{rl}gcd & \;{=} #2\\  $lb$ & \;{=} #3\end{tabular}}; }

\tikzset{BABNode/.style = {circle, fill=white, draw=black, inner sep=2pt, align=center, thick},
    labelNode/.style = {circle, fill=white, inner sep=1pt, outer sep=3pt, midway,font=\small,text=blue},
    lbNode/.style = {anchor=south east, xshift=-5pt, font=\scriptsize, fill=white, rectangle, inner sep=0pt, outer sep=0pt, align=left},
    lbNodeR/.style = {anchor=south west, xshift=3pt, font=\scriptsize, fill=white, rectangle, inner sep=0pt, outer sep=0pt, align=left},
    ubNode/.style = {anchor=north, yshift=-5pt, font=\scriptsize, fill=white, inner sep=1pt},
    pruned/.style = {draw=gray!50!white, text=gray!50!white}}

\node[BABNode] (n1-1) at ($2.5*(TreeLeafDistance)-0.0*(TreeLayerHeight)$) {};
\node[BABNode] (n2-1) at ($0.5*(TreeLeafDistance)-1.0*(TreeLayerHeight)$) {};
\node[BABNode] (n2-2) at ($2.5*(TreeLeafDistance)-1.0*(TreeLayerHeight)$) {};
\node[BABNode] (n2-3) at ($4.5*(TreeLeafDistance)-1.0*(TreeLayerHeight)$) {};
\node[BABNode, pruned] (n3-1) at ($0.0*(TreeLeafDistance)-2.0*(TreeLayerHeight)$) {};
\node[BABNode, pruned] (n3-2) at ($1.0*(TreeLeafDistance)-2.0*(TreeLayerHeight)$) {};
\node[BABNode] (n3-3) at ($2.0*(TreeLeafDistance)-2.0*(TreeLayerHeight)$) {};
\node[BABNode] (n3-4) at ($3.0*(TreeLeafDistance)-2.0*(TreeLayerHeight)$) {};
\node[BABNode, pruned] (n3-5) at ($4.0*(TreeLeafDistance)-2.0*(TreeLayerHeight)$) {}; 
\node[BABNode, pruned] (n3-6) at ($5.0*(TreeLeafDistance)-2.0*(TreeLayerHeight)$) {}; 
\node[BABNode, pruned] (n4-1) at ($0.0*(TreeLeafDistance)-3.0*(TreeLayerHeight)$) {};
\node[BABNode, pruned] (n4-2) at ($1.0*(TreeLeafDistance)-3.0*(TreeLayerHeight)$) {}; 
\node[BABNode, pruned] (n4-3) at ($2.0*(TreeLeafDistance)-3.0*(TreeLayerHeight)$) {};
\node[BABNode, pruned] (n4-4) at ($3.0*(TreeLeafDistance)-3.0*(TreeLayerHeight)$) {}; 
\node[BABNode, pruned] (n4-5) at ($4.0*(TreeLeafDistance)-3.0*(TreeLayerHeight)$) {}; 
\node[BABNode, pruned] (n4-6) at ($5.0*(TreeLeafDistance)-3.0*(TreeLayerHeight)$) {}; 

\path (n1-1) edge node[labelNode] {1} (n2-1) 
             edge node[labelNode] {2} (n2-2)
             edge node[labelNode] {4} (n2-3);
             
\path (n2-1) edge[pruned] node[labelNode] {2} (n3-1) 
             edge[pruned] node[labelNode] {4} (n3-2);
             
\path (n2-2) edge node[labelNode] {1} (n3-3) 
             edge node[labelNode] {4} (n3-4);
            
\path (n2-3) edge[pruned] node[labelNode] {1} (n3-5) 
             edge[pruned] node[labelNode] {2} (n3-6); 

             
\path (n3-1) edge[pruned] node[labelNode] {4} (n4-1);
\path (n3-2) edge[pruned] node[labelNode] {2} (n4-2);
\path (n3-3) edge[pruned] node[labelNode] {4} (n4-3);
\path (n3-4) edge[pruned] node[labelNode] {1} (n4-4);
\path (n3-5) edge[pruned] node[labelNode] {2} (n4-5);
\path (n3-6) edge[pruned] node[labelNode] {1} (n4-6);

{
\setlength{\tabcolsep}{0pt}
\renewcommand{\arraystretch}{0.75}

\nodeLblGcdLbPack{n1-1}{1}{339}{3, 3, 1}
\nodeLblGcdLbPackU{n2-1}{2}{353}{2, 4}

\node[ubNode] at (n2-1) {\uline{$ub=353$}};

\nodeLblGcdLbPack{n2-2}{1}{339}{3, 3, 1}
\nodeLblGcdLbPackU{n3-3}{4}{342}{2, 5}
\node[ubNode] at (n3-3) {\uline{$ub=342$}};

\nodeLblGcdLbR{n3-4}{1}{342}

\nodeLblGcdLbR{n2-3}{1}{360}

\draw[very thick, Clr2] ($(n2-2.south east) + (0.25,-0.15)$) rectangle ($(n2-2D.north west) + (-0.6,0.4)$);

}

\end{tikzpicture}
    \caption{Branch-and-bound tree with the bin-packing primal heuristic. Feasible solutions are found earlier in the search tree.}
    \label{fig:bab-packing}
\end{figure}


\subsection{Initial Heuristic based on the Blocks Found in the Root}\label{sec:bin_finding}
\AbbrBAB{} algorithm can initialized with an arbitrary feasible solution, allowing it to prune unpromising branches earlier in the tree.
We propose a heuristic algorithm that provides such a feasible solution based on the root node relaxation performed by the lower bound calculation described in \cref{sec:bab-lb-gcd}.

The heuristic algorithm works on the following premise.
Consider that the relaxation of the problem is solved in the root node, i.e., a lower bound sequence $\FixSeqPartialLb$ is computed.
At each node, we attempt to solve the $\ProblemBinPack$ problem, aiming to pack the jobs within the set of consecutive processing intervals $( \ProcSequence[1], \ProcSequence[2], \dots, \ProcSequence[n(\FixSeqPartialLb)] )$ that were computed for the relaxed solution $\FixSeqPartialLb$.
If such a solution is found, then its objective value matches $\TotalEnergyCost{\FixSeqPartialLb}$, and the problem is already solved at the root node.
If this is not the case, then the set of intervals $( \ProcSequence[1], \ProcSequence[2], \dots, \ProcSequence[n(\FixSeqPartialLb)] )$ cannot pack the jobs $\SetJobs$ and thus represents an infeasible instance of the $\ProblemBinPack$ problem.
Therefore, one could attempt to enlarge some of the processing intervals to accommodate all the jobs $\SetJobs$, hoping that if the enlargement is done in a sensitive way, then the resulting packing should lead to a solution of a similar quality as the optimal solution found for the relaxed problem.
We treat the enlargement of the original bins $( \ProcSequence[1], \ProcSequence[2], \dots, \ProcSequence[n(\FixSeqPartialLb)] )$ as an optimization problem that we call the \textit{Bin-Finding problem}:

\begin{sloppypar}
\paragraph{\( \ProblemBinFind \) problem} Given block sizes $b_1, \dots, b_k$ solve the following ILP problem:
\begin{align}
    & \min z \label{eq:binfind_start} \\
    & \text{subject to} \notag \\
    & z \geq b_i - s_i, \ \forall i \in \{1, \dots, k \} \\
    & z \geq s_i - b_i, \ \forall i \in \{1, \dots, k \} \\
    & s_{i} = \sum\limits_{\IdxJob \in \SetJobs} x_{i,\IdxJob} \cdot \ProcTime{\IdxJob}, \ \forall i \in \{1, \dots, k \} \\
    & \sum\limits_{i = 1}^{k} x_{i,\IdxJob} = 1, \ \forall \IdxJob \in \SetJobs{}\\
    & \text{where} \notag \\
    & z \in \mathbb{R}, \\
    & s_i \in \mathbb{R}_{\geq 0} \ \forall i \in \{1, \dots, k \},\\
    & x_{i,j}\in\{0,1\} \ \forall i \in \{1, \dots, k \}, \forall j \in \SetJobs. \label{eq:binfind_end}
\end{align}

\end{sloppypar}

The model assigns the jobs to blocks such that the maximal difference between the suggested block size $b_i$ and enlarged block size $s_i$ found is minimized.
Variable $s_i$ represents the new size of the block $i$, considering all jobs $\Job{\IdxJob}$ with $x_{i,j}=1$.  
To compute an initial upper bound for the \AbbrBAB{}, we use lengths of processing intervals $|\ProcSequence[1]|, \dots, |\ProcSequence[n(\FixSeqPartialLb)]|$ as block sizes $b_1, \ldots, b_{n(\FixSeqPartialLb)}$ obtained from the relaxed solution $\FixSeqPartialLb$ computed in the root node.
This way, we obtain new `aggregated jobs', which greatly reduce the size of the problem that can be then solved efficiently by finding an optimal schedule of \emph{aggregated jobs} with processing times $s_1, \ldots, s_{n(\FixSeqPartialLb)}$ using the job-interval graph from \cref{sec:fixseq_tec} under the sequence $(1, \ldots, n(\FixSeqPartialLb))$.




\section{Experiments}\label{sec:experiments}
In this section, we perform numerical experiments to evaluate the efficiency of the proposed \AbbrBAB{} introduced above.
Specifically, in \cref{sec:experiment_sota}, we compare \AbbrBAB{} with the
ILP model of~\cite{2018:Aghelinejad} that does not exploit optimal state switching described in \cref{sec:spaces} and the former state-of-the-art ILP model introduced in~\cite{2020a:Benedikt}.
Furthermore, in \cref{sec:experiments-real-prices}, we design larger and more complex instances with historical energy price profiles, and we investigate the effect of the different groups of the processing times.

The experiments were conducted on Intel Xeon Silver 4110 CPU 2.10 GHz processor with 100 GB of RAM. 
The construction of the ILP models is done in Python 3.7, and the models are solved using the Gurobi solver.
The algorithm for optimal state switching is written in C\# while \AbbrBAB{} algorithm is written in C++. 
Please note that only the construction of the ILP model is performed in Python, whose runtime is negligible, while the actual solution of the model is carried out using the C++ code of the Gurobi solver.
Furthermore, the C\# implementation of the optimal state switching problem uses a Just-in-Time (JIT) compiler, so its performance is close to the native code.


\setlength{\tabcolsep}{5pt}



\subsection{Comparison to the state-of-the-art}\label{sec:experiment_sota}
In this experiment, we compare our \AbbrBAB{} algorithm with two former state-of-the-art ILP models.
The first model in this comparison is \AbbrIlpRef{} proposed by \cite{2018:Aghelinejad}, which is the former state-of-the-art model that does not use pre-processing techniques described in \cref{sec:spaces}.
Later, an ILP model utilizing the interval-state graph for pre-processing, denoted as \AbbrIlpOur{}, was coined by~\cite{2020a:Benedikt} and currently represents the state-of-the-art solution method.

\begin{table}[ht]
    \centering
    \caption{Comparison of upper bound \AbbrObjective{}, lower bound \AbbrLb{} and runtime \AbbrTime{} between the algorithms.}

\renewrobustcmd{\bfseries}{\fontseries{b}\selectfont}
\fontsize{8}{9}\selectfont   

\bgroup
\def\arraystretch{0.75}

\begin{tabular}{
    S[table-format=3.0]
    S[table-format=4.0] @{\hskip \TabColSep}
    S[table-format=5.0]
    S[table-format=5.0]
    S[table-format=4.0] @{\hskip \TabColSep}
    S[table-format=5.0]
    S[table-format=5.0]
    S[table-format=4.0] @{\hskip \TabColSep}
    S[table-format=4.0]
    S[table-format=4.0]
    S[table-format=1.1] @{\hskip \TabColSep}
    S[table-format=2.1]}
    \multicolumn{12}{c}{Machine state transition diagram: \DataSingleOff{}} \\
    \toprule
    \multicolumn{2}{c @{\hskip \TabColSep}}{Instance} & 
    \multicolumn{3}{c @{\hskip \TabColSep}}{\AbbrIlpRef{}~\citep{2018:Aghelinejad}} &
    \multicolumn{3}{c @{\hskip \TabColSep}}{\AbbrIlpOur{}~\citep{2020a:Benedikt}} &
    \multicolumn{3}{c @{\hskip \TabColSep}}{\AbbrBAB{}} &
    \multicolumn{1}{c}{\AbbrPreProc{}}  \\
    \cmidrule(lr{\TabColSep}){1-2} \cmidrule(lr{\TabColSep}){3-5} \cmidrule(lr{\TabColSep}){6-8} \cmidrule(lr{\TabColSep}){9-11} \cmidrule{12-12}
    \multicolumn{1}{c}{\scriptsize $\NumJobs$} &
    \multicolumn{1}{c @{\hskip \TabColSep}}{\scriptsize $\NumIntervals$} &
    \multicolumn{1}{c}{\scriptsize \AbbrObjective{} [-]} & \multicolumn{1}{c}{\scriptsize \AbbrLb{} [-]} & \multicolumn{1}{c @{\hskip \TabColSep}}{\scriptsize \AbbrTime{} [\si{\second}]} &
    \multicolumn{1}{c}{\scriptsize \AbbrObjective{} [-]} & \multicolumn{1}{c}{\scriptsize \AbbrLb{} [-]} & \multicolumn{1}{c @{\hskip \TabColSep}}{\scriptsize \AbbrTime{} [\si{\second}]} &
    \multicolumn{1}{c}{\scriptsize \AbbrObjective{} [-]} & \multicolumn{1}{c}{\scriptsize \AbbrLb{} [-]} & \multicolumn{1}{c @{\hskip \TabColSep}}{\scriptsize \AbbrTime{} [\si{\second}]} &
    \multicolumn{1}{c}{\scriptsize \AbbrTime{} [\si{\second}]} \\
    \midrule
    150 & 527 & \Optimum{ 8582} &  8567 & \AbbrTimeLimit{} & \Optimum{ 8582} & \Optimum{ 8582} &  188 & \Optimum{ 8582} & \Optimum{ 8582} & 0.6 & 1.0 \\
150 & 647 &  8726 &  8247 & \AbbrTimeLimit{} & \Optimum{ 8409} & \Optimum{ 8409} &  279 & \Optimum{ 8409} & \Optimum{ 8409} & 0.9 & 2.9 \\
150 & 767 &  8557 &  7787 & \AbbrTimeLimit{} & \Optimum{ 8132} & \Optimum{ 8132} &  619 & \Optimum{ 8132} & \Optimum{ 8132} & 0.8 & 5.5 \\
150 & 888 &  8976 &  6780 & \AbbrTimeLimit{} & \Optimum{ 8078} & \Optimum{ 8078} &  521 & \Optimum{ 8078} & \Optimum{ 8078} & 1.1 & 9.1 \\
170 & 650 & 10596 &  9628 & \AbbrTimeLimit{} & \Optimum{10068} & \Optimum{10068} &  289 & \Optimum{10068} & \Optimum{10068} & 0.7 & 2.3 \\
170 & 799 & 10794 &  8832 & \AbbrTimeLimit{} & \Optimum{ 9820} & \Optimum{ 9820} &  939 & \Optimum{ 9820} & \Optimum{ 9820} & 0.8 & 4.6 \\
170 & 948 & 10940 &  8343 & \AbbrTimeLimit{} & \Optimum{ 9637} & \Optimum{ 9637} &  822 & \Optimum{ 9637} & \Optimum{ 9637} & 1.7 & 9.3 \\
170 & 1097 & 11189 &  8124 & \AbbrTimeLimit{} & \Optimum{ 9620} & \Optimum{ 9620} & 1281 & \Optimum{ 9620} & \Optimum{ 9620} & 2.4 & 13.4 \\
190 & 757 & 12555 & 11206 & \AbbrTimeLimit{} & \Optimum{12008} & \Optimum{12008} &  243 & \Optimum{12008} & \Optimum{12008} & 0.7 & 3.9 \\
190 & 930 & 12882 & 10521 & \AbbrTimeLimit{} & \Optimum{11758} & \Optimum{11758} &  951 & \Optimum{11758} & \Optimum{11758} & 1.1 & 6.9 \\
190 & 1104 & 12791 &  9949 & \AbbrTimeLimit{} & \Optimum{11611} & \Optimum{11611} & 3095 & \Optimum{11611} & \Optimum{11611} & 2.4 & 13.3 \\
190 & 1277 & 12757 &     0 & \AbbrTimeLimit{} & \Optimum{11465} & \Optimum{11465} & 1319 & \Optimum{11465} & \Optimum{11465} & 2.0 & 22.7 \\
\midrule
\multicolumn{4}{l}{Average time [\si{\second}]:} & {\textgreater 3600} & & &  879 & & & 1.3 & 7.9 \\
\multicolumn{4}{l}{Average optimality gap [\si{\percent}]:} & {7.58} & & & {0.00} & & & {0.00} \\
\bottomrule
    \\
    \multicolumn{12}{c}{Machine state transition diagram: \DataMultiSby{}} \\
    \toprule
    \multicolumn{2}{c @{\hskip \TabColSep}}{Instance} & 
    \multicolumn{3}{c @{\hskip \TabColSep}}{\AbbrIlpRef{}~\citep{2018:Aghelinejad}} &
    \multicolumn{3}{c @{\hskip \TabColSep}}{\AbbrIlpOur{}~~\citep{2020a:Benedikt}} &
    \multicolumn{3}{c @{\hskip \TabColSep}}{\AbbrBAB{}} &
    \multicolumn{1}{c}{\AbbrPreProc{}}  \\
    \cmidrule(lr{\TabColSep}){1-2} \cmidrule(lr{\TabColSep}){3-5} \cmidrule(lr{\TabColSep}){6-8} \cmidrule(lr{\TabColSep}){9-11} \cmidrule{12-12}
    \multicolumn{1}{c}{\scriptsize $\NumJobs$} &
    \multicolumn{1}{c @{\hskip \TabColSep}}{\scriptsize $\NumIntervals$} &
    \multicolumn{1}{c}{\scriptsize \AbbrObjective{} [-]} & \multicolumn{1}{c}{\scriptsize \AbbrLb{} [-]} & \multicolumn{1}{c @{\hskip \TabColSep}}{\scriptsize \AbbrTime{} [\si{\second}]} &
    \multicolumn{1}{c}{\scriptsize \AbbrObjective{} [-]} & \multicolumn{1}{c}{\scriptsize \AbbrLb{} [-]} & \multicolumn{1}{c @{\hskip \TabColSep}}{\scriptsize \AbbrTime{} [\si{\second}]} &
    \multicolumn{1}{c}{\scriptsize \AbbrObjective{} [-]} & \multicolumn{1}{c}{\scriptsize \AbbrLb{} [-]} & \multicolumn{1}{c @{\hskip \TabColSep}}{\scriptsize \AbbrTime{} [\si{\second}]} &
    \multicolumn{1}{c}{\scriptsize \AbbrTime{} [\si{\second}]} \\
    \midrule
    150 & 529 & \Optimum{21910} & 21562 & \AbbrTimeLimit{} & \Optimum{21910} & \Optimum{21910} &  133 & \Optimum{21910} & \Optimum{21910} & 0.9 & 1.1 \\
150 & 649 & 29425 & 20685 & \AbbrTimeLimit{} & \Optimum{21821} & \Optimum{21821} &  705 & \Optimum{21821} & \Optimum{21821} & 1.2 & 3.1 \\
150 & 769 & 37764 & 18140 & \AbbrTimeLimit{} & \Optimum{21353} & \Optimum{21353} &  957 & \Optimum{21353} & \Optimum{21353} & 1.3 & 5.2 \\
150 & 890 & 43929 & 16799 & \AbbrTimeLimit{} & \Optimum{21266} & \Optimum{21266} &  718 & \Optimum{21266} & \Optimum{21266} & 1.4 & 8.5 \\
170 & 651 & 27862 & 25015 & \AbbrTimeLimit{} & \Optimum{25807} & \Optimum{25807} &  811 & \Optimum{25807} & \Optimum{25807} & 0.8 & 2.6 \\
170 & 799 & 39095 & 21981 & \AbbrTimeLimit{} & \Optimum{25518} & \Optimum{25518} & 1250 & \Optimum{25518} & \Optimum{25518} & 1.1 & 5.0 \\
170 & 948 & 46083 & 20709 & \AbbrTimeLimit{} & \Optimum{25279} & \Optimum{25279} & 2925 & \Optimum{25279} & \Optimum{25279} & 2.2 & 8.5 \\
170 & 1096 & 53177 & 20091 & \AbbrTimeLimit{} & \Optimum{25279} & \Optimum{25279} & 2088 & \Optimum{25279} & \Optimum{25279} & 2.0 & 14.1 \\
190 & 756 & 38471 & 27984 & \AbbrTimeLimit{} & \Optimum{30563} & \Optimum{30563} &  808 & \Optimum{30563} & \Optimum{30563} & 0.9 & 4.2 \\
190 & 929 & 46319 & 26166 & \AbbrTimeLimit{} & \Optimum{30224} & \Optimum{30224} & 1086 & \Optimum{30224} & \Optimum{30224} & 2.3 & 7.5 \\
190 & 1102 & 53751 & 24630 & \AbbrTimeLimit{} & \Optimum{30224} & \Optimum{30224} & 2050 & \Optimum{30224} & \Optimum{30224} & 1.9 & 13.6 \\
190 & 1275 & 61547 &     0 & \AbbrTimeLimit{} & \Optimum{30071} & \Optimum{30071} & 2490 & \Optimum{30071} & \Optimum{30071} & 4.5 & 23.7 \\
\midrule
\multicolumn{4}{l}{Average time [\si{\second}]:} & {\textgreater 3600} & & & 1335 & & & 1.7 & 8.1 \\
\multicolumn{4}{l}{Average optimality gap [\si{\percent}]:} & {34.23} & & & {0.00} & & & {0.00} \\
\bottomrule
    \end{tabular}

\egroup

    \label{tab:exp1-large}
\end{table}
For the comparison, we used instances with two different machine state diagrams.
First, the NOSBY (i.e., no stand-by mode, depicted in \cref{fig:example-func-power-time}) instances were proposed by \cite{2014:Shrouf} and correspond to the state diagrams without a standby mode.
Second, more complex instances with two different stand-by modes are denoted as TWOSBY, containing a total of 5 states~\cite[p.~12]{2020a:Benedikt}.
For each type of machine state diagram, we generate instances with $n\in\{150,170,190\}$ jobs, as this is roughly the frontier of the limit of efficient computations of ILP models \AbbrIlpRef{}~\citep{2018:Aghelinejad} and \AbbrIlpOur{}~\citep{2020a:Benedikt}.
The scheme for the generation of the dataset follows standard benchmarks~\citep{2014:Shrouf,2018:Aghelinejad}, i.e., processing times being drawn from the discrete uniform distribution $\ProcTime{\IdxJob}\sim\mathcal{U}[1,5]$ and the energy costs in each interval follows $c_i\sim\mathcal{U}[1,10]$ distribution.
Note that even though processing times $\ProcTime{\IdxJob}$ may seem very small, they, in fact, represent the number of intervals that are typically 15 minutes long, thus resembling a realistic benchmark. 
The number of intervals $\SetIntervals = \{\Interval{1}, \Interval{2}, \dots, \Interval{\NumIntervals}\}$ in each instance is given as $\lambda\cdot\left(T(\StateOff,\StateProc)+\sum_{\IdxJob\in\SetJobs} \ProcTime{\IdxJob}+T(\StateProc,\StateOff)\right)$ where the term inside the brackets represents a lower bound on horizon length $\NumIntervals$ (measured in the number of equidistant intervals) and $\lambda\in\{1.3,1.6,1.9,2.2\}$ is a scaling factor to inflate the scheduling horizon and to provide a room for optimization. 

The results are reported in \cref{tab:exp1-large}.
Each row corresponds to one problem instance. 
In the first column \textit{Instance}, we report the specific parameters of the particular instance, and the following three columns account for the state-of-the-art methods included in the comparison.
The last column \AbbrPreProc{} denotes the time spent in the pre-processing step involving the shortest path algorithm on the interval-state graph, as described in \cref{sec:spaces}.
This time, although negligible, needs to be included in the runtimes of algorithms \AbbrIlpOur{} and \AbbrBAB{} as these two utilize the pre-processing step.
The time limit for one instance was set to \Second{3600}, \AbbrTimeLimit{} stands for the time limit reached. Optimal bound values are denoted in bold.

The results clearly demonstrate the power of the pre-processing step---for the NOSBY state transition diagram, the model \AbbrIlpRef{}~\citep{2018:Aghelinejad} cannot solve the considered instances up to optimality, reaching an average proved optimality gap 7.58\%, whereas the algorithms \AbbrIlpOur{} and \AbbrBAB{} utilizing the pre-processing (\AbbrPreProc{}) solve all instances up to optimality.
More significantly, the number of states $|\SetStates|$ in the machine state transition diagram heavily affects the performance of \AbbrIlpRef{}.
For example, in instances TWOSBY, the average optimality gap of \AbbrIlpRef{} increases to 34.23\%, while the other algorithms are still able to solve all instances up to optimality within the time limit.
However, it can be seen that compiling all possible state transitions into the optimal switching problem slightly increases the complexity of the resulting problem instance, reflected in increased runtime of \AbbrIlpOur{} wheres \AbbrBAB{} is largely unaffected.

Comparing the runtimes, it can be seen that our proposed algorithm \AbbrBAB{} clearly improves on the result of \AbbrIlpOur{} about a factor of 10-100$\times$.
This demonstrates the efficiency of utilizing the bin-packing structure of the problem, heavily exploited by \AbbrBAB{}.
However, we note that the benchmark results reveal that, in fact, the existing strategy for dataset generation does not provide instances that are challenging enough since all of the instances were solved in less than \Second{10} by \AbbrBAB{}.
Therefore, in the following sections, we focus on understanding which combination of instance parameters makes the problem difficult to solve, and we explore the empirical complexity of the instances utilizing real-life energy prices.


\subsection{Processing time groups with real energy costs}\label{sec:experiments-real-prices}
In this section, we study the effect of different groups of jobs' processing times on the computational times required by \AbbrIlpOur{} model and \AbbrBAB{} algorithm.
In \cref{sec:experiment_sota}, we worked with the standard benchmark instances that assume the set of possible processing times is~$\{1,2,3,4,5\}$~\citep{2014:Shrouf}.
Here, we also test different groups of processing times, for example, all odd $\{2,4,6,8,10\}$, two primes $\{3,7\}$ or primes with unit job $\{1,2,3,5,7\}$.
All tested processing time groups can be seen in \cref{tab:exp2-aggregation}.
The state diagram of the resource used is TWOSBY (i.e., two stand-by states, five states in total, see description in \cref{sec:experiment_sota}).
In addition, we also use real energy (electricity) prices extracted from the independent market operator (OTE) in the Czech Republic.
A part of such profile that was used is shown in \cref{fig:example-profile}.

\begin{figure}
    \centering
    \pgfplotstableread[col sep=comma]{images_sources/real_energy_costs_example.csv}\datatable
    \begin{tikzpicture}
        \begin{axis}[
            width=0.75\textwidth,
            height=0.25\textwidth,
            ybar, 
            label style={font=\small},
            tick label style={font=\scriptsize},
            xtick={0,23},
            xticklabels={Fri 12AM, Sat 12AM},
            ylabel={cost $\EnergyCost{\IdxInterval}$ [CZK]},
            xlabel={interval $\IdxInterval$ [-]},
            ymajorgrids,
            bar width=3pt,
            enlarge x limits={abs=0.75cm}, 
        ]
        \addplot table[x=idx, y=cost] {\datatable};
        \end{axis}
    \end{tikzpicture}
    \caption{Example of a 48-hour part of the energy profile with real electricity costs.}
    \label{fig:example-profile}
\end{figure}
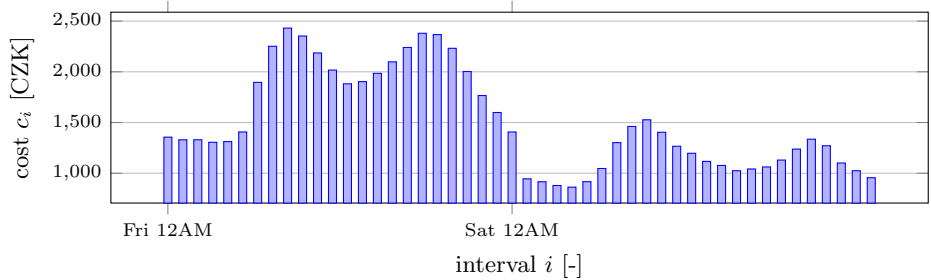


\begin{table}[ht]
    \centering
    \caption{Comparison between  \AbbrIlpOur{} and \AbbrBAB{} with respect to different processing times groups and numbers of jobs; \AbbrTime{} denotes the average runtime, \#o is the number of optimally solved instances, \#s is the number of instances for which a solution was found but was not proven optimal, and gap is the optimality gap for instances counted in \#s column; \AbbrPreProc{} denotes the pre-processing used for both compared methods.}
    \begin{adjustbox}{max width=\textwidth}

\renewrobustcmd{\bfseries}{\fontseries{b}\selectfont}
\fontsize{8}{9}\selectfont   

\bgroup
\def\arraystretch{0.75}

\begin{tabular}{
    S[table-format=3.0]
    l @{\hskip \TabColSep}
    S[table-format=3.2]
    S[table-format=2.0]
    S[table-format=2.0]
    S[table-format=1.3] @{\hskip \TabColSep}
    S[table-format=3.2]
    S[table-format=2.0]
    S[table-format=2.0]
    S[table-format=1.3] @{\hskip \TabColSep}
    S[table-format=2.1]}
    \toprule
     \multicolumn{2}{c @{\hskip \TabColSep}}{Instance} & 
    \multicolumn{4}{c @{\hskip \TabColSep}}{\AbbrIlpOur{}~\citep{2020a:Benedikt}}  & \multicolumn{4}{c @{\hskip \TabColSep}}{\AbbrBAB{}} &
    \multicolumn{1}{c}{\AbbrPreProc{}} \\
    \cmidrule(lr{\TabColSep}){1-2} \cmidrule(lr{\TabColSep}){3-6} \cmidrule(lr{\TabColSep}){7-10} \cmidrule{11-11}
    \multicolumn{1}{c}{\scriptsize $\NumJobs$} &
    \multicolumn{1}{c @{\hskip \TabColSep}}{\scriptsize processing times} &
    \multicolumn{1}{c}{\scriptsize \AbbrTime{} [\si{\second}]} & \multicolumn{1}{c}{\scriptsize \#o [-]} & \multicolumn{1}{c}{\scriptsize \#s [-]} & \multicolumn{1}{c @{\hskip \TabColSep}}{\scriptsize gap [\%]} &
    \multicolumn{1}{c}{\scriptsize \AbbrTime{} [\si{\second}]} & \multicolumn{1}{c}{\scriptsize \#o [-]} & \multicolumn{1}{c}{\scriptsize \#s [-]} & \multicolumn{1}{c @{\hskip \TabColSep}}{\scriptsize gap [\%]} &
    \multicolumn{1}{c}{\scriptsize \AbbrTime{} [\si{\second}]} \\
    \midrule
    

    50 & $\{1, 2, 3, 4, 5, 6, 7, 8, 9, 10\}$ & 79.13 & 20 & 0 & \textendash & 0.21 & 20 & 0 & \textendash & 0.4 \\
& $\{1, 2, 3, 5, 7\}$ & 21.05 & 20 & 0 & \textendash & 0.18 & 20 & 0 & \textendash & 0.1 \\
& $\{2, 4, 6, 8, 10\}$ & 52.72 & 20 & 0 & \textendash & 0.21 & 20 & 0 & \textendash & 0.5 \\
& $\{2, 4\}$ & 6.31 & 20 & 0 & \textendash & 0.18 & 20 & 0 & \textendash & 0.1 \\
& $\{3, 5, 6, 7\}$ & 46.00 & 20 & 0 & \textendash & 0.24 & 20 & 0 & \textendash & 0.4 \\
& $\{3, 7\}$ & 46.17 & 20 & 0 & \textendash & 0.24 & 20 & 0 & \textendash & 0.3 \\
& $\{8, 10\}$ & 77.35 & 20 & 0 & \textendash & 0.33 & 20 & 0 & \textendash & 1.6 \\
100 & $\{1, 2, 3, 4, 5, 6, 7, 8, 9, 10\}$ & 444.72 & 11 & 9 & 0.16 & 0.48 & 20 & 0 & \textendash & 3.0 \\
& $\{1, 2, 3, 5, 7\}$ & 134.42 & 20 & 0 & \textendash & 0.29 & 20 & 0 & \textendash & 0.8 \\
& $\{2, 4, 6, 8, 10\}$ & 315.68 & 17 & 2 & 0.01 & 0.48 & 20 & 0 & \textendash & 3.7 \\
& $\{2, 4\}$ & 136.14 & 20 & 0 & \textendash & 0.24 & 20 & 0 & \textendash & 0.5 \\
& $\{3, 5, 6, 7\}$ & 471.05 & 9 & 11 & 0.19 & 0.44 & 20 & 0 & \textendash & 2.9 \\
& $\{3, 7\}$ & 483.35 & 7 & 13 & 0.16 & 0.40 & 20 & 0 & \textendash & 2.5 \\
& $\{8, 10\}$ & 525.80 & 5 & 15 & 0.07 & 0.89 & 20 & 0 & \textendash & 10.6 \\
150 & $\{1, 2, 3, 4, 5, 6, 7, 8, 9, 10\}$ & 543.10 & 4 & 11 & 1.23 & 0.76 & 20 & 0 & \textendash & 7.9 \\
& $\{1, 2, 3, 5, 7\}$ & 466.84 & 9 & 11 & 0.10 & 0.43 & 20 & 0 & \textendash & 3.1 \\
& $\{2, 4, 6, 8, 10\}$ & 514.14 & 8 & 12 & 0.29 & 0.91 & 20 & 0 & \textendash & 10.4 \\
& $\{2, 4\}$ & 278.49 & 14 & 6 & 0.03 & 0.40 & 20 & 0 & \textendash & 1.8 \\
& $\{3, 5, 6, 7\}$ & 539.66 & 4 & 11 & 0.85 & 0.73 & 20 & 0 & \textendash & 6.8 \\
& $\{3, 7\}$ & 581.41 & 2 & 18 & 0.29 & 0.69 & 20 & 0 & \textendash & 6.1 \\
& $\{8, 10\}$ & \AbbrTimeLimit{} & 0 & 0 & \textendash & 143.74 & 15 & 5 & 0.02 & 31.8 \\
200 & $\{1, 2, 3, 4, 5, 6, 7, 8, 9, 10\}$ & \AbbrTimeLimit{} & 0 & 7 & 7.66 & 1.30 & 20 & 0 & \textendash & 21.1 \\
& $\{1, 2, 3, 5, 7\}$ & 565.26 & 2 & 18 & 0.61 & 0.61 & 20 & 0 & \textendash & 5.3 \\
& $\{2, 4, 6, 8, 10\}$ & 577.29 & 1 & 3 & 9.39 & 1.46 & 20 & 0 & \textendash & 25.2 \\
& $\{2, 4\}$ & 425.80 & 12 & 8 & 0.15 & 0.52 & 20 & 0 & \textendash & 3.7 \\
& $\{3, 5, 6, 7\}$ & 576.54 & 1 & 15 & 2.74 & 1.17 & 20 & 0 & \textendash & 19.7 \\
& $\{3, 7\}$ & \AbbrTimeLimit{} & 0 & 19 & 1.32 & 1.09 & 20 & 0 & \textendash & 17.7 \\
& $\{8, 10\}$ & \AbbrTimeLimit{} & 0 & 0 & \textendash & 132.99 & 15 & 5 & 0.02 & 78.7 \\    
    \bottomrule
    \end{tabular}
\egroup

\end{adjustbox}

    \label{tab:exp2-aggregation}
\end{table}

The aggregated results for all instances with a fixed number of jobs $n$ and processing time groups can be seen in \cref{tab:exp2-aggregation}.
First, note that ILP model \AbbrIlpOur{} solves all instances with $n=50$ jobs up to optimality but cannot solve all instances with $n=100$ jobs.
This also shows that former standard benchmarks used in \cref{sec:experiment_sota} were not that challenging since there \AbbrIlpOur{} was able to solve instances with $n=190$.
Next, the results also suggest that the set of jobs' processing times has a significant effect on the complexity of the instances, even for \AbbrIlpOur{}.
For example, with $n=100$ jobs, all instances with $\ProcTime{\IdxJob}\in\{2,4\}$ were solved up to optimality, wheres for $\ProcTime{\IdxJob}\in\{3,7\}$ just 7 out of 20.
The computational times of \AbbrBAB{} remain almost unaffected, except the group of processing times $\{8,10\}$ with $n\in\{150,200\}$ jobs,  where 10 instances could not be solved up to the optimality. 
This surprising difficulty of these instances is also reflected in the performance of \AbbrIlpOur{} that could not find a feasible solution for any instance with $\ProcTime{\IdxJob}\in\{8,10\}$.
The detailed look on the instances where \AbbrBAB{} could not prove the optimality is given in \cref{tab:exp2-detail}.
There, for each specific unsolved instance, we see the best objective found~\AbbrObjective{}, best proven lower bound~\AbbrLb{}, and a number of expanded nodes by \AbbrBAB{}.
In all cases, the found solutions are typically provably within 0.02\% of the optimality.
In any case, the observed difficulty of these instances is quite remarkable---for example, in instance id=348, the absolute gap between the proven lower bound and the best solution found was only 19 units of cost.
We assume that one of the sources of the difficulty would be the fact that $\{8,10\}$ are two of the larger processing times used; thus, the relaxation procedure in lower bound computation from \cref{sec:bab-lb-gcd} has larger relaxation gap on average, even though $\gcd\{8,10\}=2>1$.
On the other hand, instances with $\ProcTime{\IdxJob}\in\{1,2,\ldots,10\}$ were solved up to the optimality even when $\gcd\{1,2,\ldots,10\}=1$.
It appears that these instances are, in fact, easier since larger variability in processing times allows us to solve the $\ProblemBinPack$ problem involved in primal upper bound computation more frequently, thus obtaining primal solutions matching the spaces produced by the lower bound procedure and closing the search subtrees, which is a phenomenon described further in the following section.
\begin{table}[ht]
    \centering
    \caption{Detailed view the instances that were not solved by \AbbrBAB{} to optimality.}
    
\begin{adjustbox}{max width=\textwidth}

\renewrobustcmd{\bfseries}{\fontseries{b}\selectfont}
\fontsize{8}{9}\selectfont   

\bgroup
\def\arraystretch{0.75}

\begin{tabular}{
    S[table-format=3.0]
    S[table-format=3.0]
    l @{\hskip \TabColSep}
    S[table-format=8.0]
    S[table-format=8.0]
    S[table-format=1.3]
    S[table-format=5.0]
    S[table-format=1.0]
    S[table-format=1.0]
    S[table-format=1.0]}
    \toprule
    \multicolumn{3}{c @{\hskip \TabColSep}}{Instance} & 
    \multicolumn{4}{c}{\AbbrBAB{}} \\
    \cmidrule(lr{\TabColSep}){1-3} \cmidrule{4-7}
    \multicolumn{1}{c}{\scriptsize id} &
    \multicolumn{1}{c}{\scriptsize $\NumJobs$} &
    \multicolumn{1}{c @{\hskip \TabColSep}}{\scriptsize proc. times} &
    \multicolumn{1}{c}{\scriptsize \AbbrObjective{} [-]} &
    \multicolumn{1}{c}{\scriptsize \AbbrLb{} [-]} &
    \multicolumn{1}{c}{\scriptsize gap [\%]} & 
    \multicolumn{1}{c}{\scriptsize \#nodes [-]} \\ 
    \midrule
    292 & 150 & $\{8, 10\}$ & 12352437 & 12349218 & 0.026 & 57457 \\
    320 & 150 & $\{8, 10\}$ & 12330437 & 12328008 & 0.020 & 60355 \\
    348 & 150 & $\{8, 10\}$ & 12288440 & 12288421 & 0.000 & 52854 \\
    376 & 150 & $\{8, 10\}$ & 12330437 & 12328008 & 0.020 & 61988 \\
    404 & 150 & $\{8, 10\}$ & 12074142 & 12071713 & 0.020 & 55328 \\
    432 & 200 & $\{8, 10\}$ & 16413555 & 16411549 & 0.012 & 27259 \\
    460 & 200 & $\{8, 10\}$ & 16293105 & 16290676 & 0.015 & 30450\\
    488 & 200 & $\{8, 10\}$ & 16433290 & 16430776 & 0.015 & 20572 \\
    516 & 200 & $\{8, 10\}$ & 16123137 & 16120217 & 0.018 & 26673\\
    544 & 200 & $\{8, 10\}$ & 16554373 & 16551944 & 0.015 & 25287\\
    \bottomrule
    \end{tabular}
\egroup

\end{adjustbox}
    \label{tab:exp2-detail}
\end{table}

\subsection{Instance hardness}\label{sec:experiments-hardness}

The varying difficulty of instances influenced by the specific processing time group revealed in \cref{tab:exp2-aggregation} motivates us to study average-case instance complexity in more detail.
We created a new set of instances under the same protocol as in~\cref{sec:experiments-real-prices}, however this time with $\NumJobs\in\{100,150,200\}$ and the following new groups of processing times: $\{9,10\}, \{8,9,10\}, \{8,9\}, \{7,8,9,10\}, \{7,8\}, \{10\},\{7,9\},\{1,2,10\}$. 
For each value of $n$ and a processing time group, we generated 20 instances.
The resulting runtimes of \AbbrBAB{} with a time limit of 10 minutes are displayed in~\cref{fig:ex_proctimes} with the logaritmic scale of the vertical axis. 
Each color represents one group of processing times, where all 60 instances belonging to each group are ordered with the increasing length of the horizon $h$, i.e., the number of intervals.


\begin{figure}[ht]
    \centering
    \pgfplotstableread[col sep=semicolon]{images_sources/hardness_sorted_by_groups_1,2,10.csv}\datatablea

    \pgfplotstableread[col sep=semicolon]{images_sources/hardness_sorted_by_groups_7,9.csv}\datatableb

    \pgfplotstableread[col sep=semicolon]{images_sources/hardness_sorted_by_groups_7,8,9,10.csv}\datatablec

     \pgfplotstableread[col sep=semicolon]{images_sources/hardness_sorted_by_groups_8,9.csv}\datatabled
     
     \pgfplotstableread[col sep=semicolon]{images_sources/hardness_sorted_by_groups_8,9,10.csv}\datatablee

     \pgfplotstableread[col sep=semicolon]{images_sources/hardness_sorted_by_groups_9,10.csv}\datatablef

     \pgfplotstableread[col sep=semicolon]{images_sources/hardness_sorted_by_groups_10.csv}\datatableg
    \begin{tikzpicture}
        \begin{axis}[
            width=0.9\textwidth,
            height=0.28\textwidth,
            ybar, 
            ymode=log,
            log origin=infty,
            label style={font=\small},
            tick label style={font=\scriptsize},
            xtick={20,85,145,210,273,335,400},
            xticklabels={${\{1,2,3\}}$,{$\{7,8,9,10\}$},{$\{7,9\}$},{$\{8,9\}$},{$\{8,9,10\}$},{$\{9,10\}$},{$\{10\}$}},
            ylabel={time [s]},
            xlabel={instance [-]},
            ymajorgrids,
            bar width=0.1pt,
            legend entries =  {{}{$\{1,2,10\}$},{}{$\{7,8,9,10\}$},  {}{$\{7,9\}$},  {}{$\{8,9\}$}, {}{$\{8,9,10\}$}, {}{$\{9,10\}$}, {}{$\{10\}$}},
            legend pos=outer north east,
            legend cell align=left
        ]
        \addplot+[opacity=0.4] table[y=RunningTimeSec] {\datatablea};
        \addplot+[opacity=0.4] table[x expr=\thisrowno{0}+60, y=RunningTimeSec] {\datatablec};
        \addplot+[opacity=0.4] table[x expr=\thisrowno{0}+120, y=RunningTimeSec] {\datatableb};
        \addplot+[opacity=0.4] table[x expr=\thisrowno{0}+180, y=RunningTimeSec] {\datatabled};
        \addplot+[opacity=0.4] table[x expr=\thisrowno{0}+240, y=RunningTimeSec] {\datatablee};
        \addplot+[opacity=0.4] table[x expr=\thisrowno{0}+300, y=RunningTimeSec] {\datatablef};
        \addplot+[opacity=0.4,color=magenta] table[x expr=\thisrowno{0}+360, y=RunningTimeSec] {\datatableg};

        \end{axis}
    \end{tikzpicture}
    \caption{Running times sorted by horizon length, split by processing time group for $n\in\{100,150,200\}$ jobs.}
    \label{fig:ex_proctimes}
\end{figure}
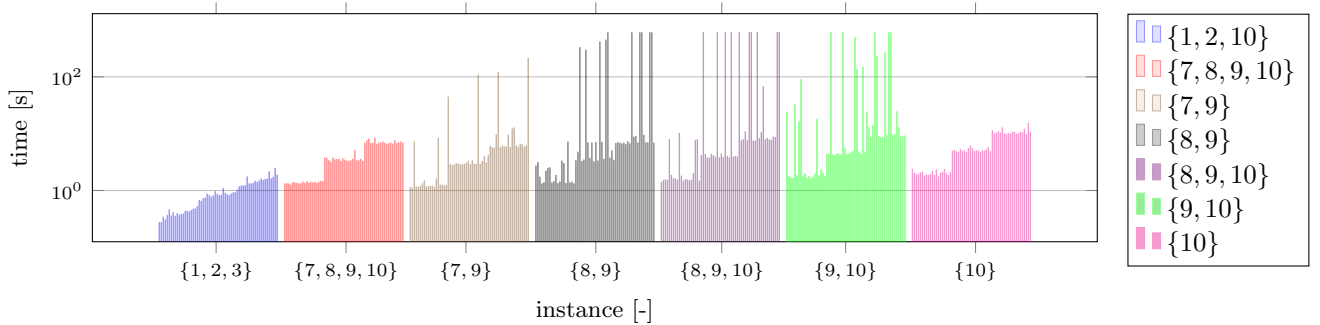

Under visual inspection, the dependence of the runtime on the length of the horizon $h$ can be seen in the form of a staircase function.
This is due to how the instances are generated, i.e., $h=1.3$ times the sum of all processing times. 
Thus, the three steps visible in each group correspond to instances with $n=100$, 150, and 200 jobs.
However, when focusing on the instance group containing processing times $8,9$ and 10, we observe more frequently the sudden spikes in the runtime with the greatest intensity around processing time group $\{8,9,10\}$.
The exact proportion of the instances that were not solved to the optimality (i.e., time limit reached -- TLR) is shown in \cref{fig:ex_timeout_hist}.

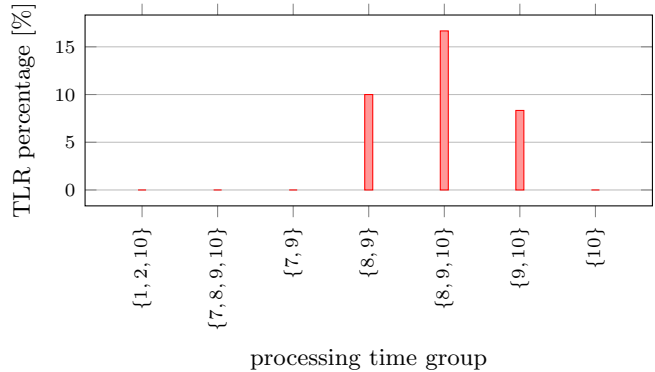
\begin{figure}[ht]
    \centering
    \pgfplotstableread[col sep=semicolon]{images_sources/timeout_counts_hist_reverse.csv}\datatable
    \begin{tikzpicture}
        \begin{axis}[
            width=0.55\textwidth,
            height=0.25\textwidth,
            ybar, 
            label style={font=\small},
            tick label style={font=\scriptsize},
            xtick=data,
            xticklabels={{}{$\{1,2,10\}$},{}{$\{7,8,9,10\}$},{}{$\{7,9\}$},{}{$\{8,9\}$}, {}{$\{8,9,10\}$}, {}{$\{9,10\}$},{}{$\{10\}$}},
            xticklabel style={rotate=90},
            ylabel={TLR percentage [\%]},
            xlabel={processing time group},
            ymajorgrids,
            bar width=3pt,
            enlarge x limits={abs=0.75cm}, 
        ]
        \addplot[fill=red!40,draw=red] table[x index=0,y=TimeoutPercentage] {\datatable};
        \end{axis}
    \end{tikzpicture}
    \caption{Proportion of instances that were not solved up to the optimality.}
    \label{fig:ex_timeout_hist}
\end{figure}

It appears to us that this phenomenon has an explanation following from a bin-packing structure of the problem.
The real-life energy price profiles share a structure with more or less periodically repeating areas with lower average costs, interleaved by areas with intervals having higher average costs, as can be seen in \cref{fig:example-profile}.
The algorithm that optimizes TEC avoids areas with high-priced intervals.
Depending on the specific values of transition time and power functions of the machine, the solution switches between off-peak intervals, in which the resource is typically in \StateProc{} state, and peak intervals, where the resource is often in \StateIdle{}/\StateOff{} states.
The sets of consecutive intervals where the machine is in \StateProc{} state then form bins, as it was introduced in \cref{sec:binpack_primal_heur}.
The combinatorial nature of the problem is then hidden within the task of how to fit jobs with given processing times into the bins.
This aspect shows that the considered problem has deeper connections to number partitioning and bin packing problems.

Instance hardness in the context of the average-case complexity of the number partition problem~(NPP) was studied by \cite{mertens2003easiesthardproblemnumber}.
Mertens shows that NPP admits a sharp phase transition of the instances, where we quickly transition from instances that almost surely have a perfect partitioning to instances that very likely do not have one, which influences the empirical runtime of the algorithm.
This transition is not directly related to the amount of the number itself, but to the specific structure of the numbers.
Paradoxically, even larger problems may become easier to solve.
Our problem better matches the similar analysis for parallel machine scheduling problem by \cite{bauke2003phase}.
For $P||C_\text{max}$ or bin-packing problem, it is notoriously known that, e.g., the presence of short jobs (items) leads to easier instances since it is easy to fill up the remaining space by the short jobs. 
The same applies to our problem --- having larger variability in jobs' processing times, especially combined with the presence of short jobs, increases the probability of perfect partitions, i.e., situations where primal upper bound heuristic based on a packing problem matches the bins proposed by the lower bound, allowing to close nodes in the search tree earlier.
In our setting, the phase transition appears to be around the group $\{8,9,10\}$, where the probability of encountering a hard instance increases.
On the opposite extreme, when all processing times are equal (i.e., the group $\ProcTime{\IdxJob}\in\{10\}$), all jobs become interchangeable, and since $\gcd=10$, the problem instances are easy again.
Therefore, we conclude that the number of jobs $n$ in the instance is not the main limiting factor of \AbbrBAB{} algorithm, but rather it is the presence of long jobs with low processing time variability within the group.

\section{Conclusion}



In this work, we addressed the problem of total energy cost minimization on a single machine with different states under time-of-use tariffs.
This problem, arising in production scheduling with energy-intensive manufacturing operations, plays a critical role in reducing energy bills.
The importance of the problem was recognized by \cite{2014:Shrouf}, who coined the problem with several authors to follow \citep{2018:Aghelinejad,aghe2019toustates,2020a:Benedikt} to follow.
The breakthrough came when \cite{2020a:Benedikt} showed that one can simplify state transitions of the machine by solving the optimal switching problem and use precomputed switching between any pair of intervals inside an ILP model without an explicit modeling.

In this work, we took this idea one step further. 
We realized that the spaces suggested by the optimal switching problem could be seen as a kind of bins to which we allocate the jobs.
This led to the design of a branch-and-bound algorithm \AbbrBAB{} that utilizes the bin packing structure by applying initial and primal heuristics and also a lower bound computation based on the polynomial-time solution of the problem for the fixed order of the jobs.
We improved on the approach proposed in~\cite{aghe2019toustates} by incorporating only the optimal switching costs, leading to a much smaller job-interval graph, therefore enabling fast lower bound computations.
These insights and improvements led to achieving a new state-of-the-art performance on the standard benchmark set by two orders of magnitude, enabling the solutions of larger instances than ever possible.
Furthermore, we have observed that the standard benchmark set is extremely well-suited for the bin packing structure utilized by \AbbrBAB{}; thus, we also designed new, much harder benchmarks considering true historical energy price profiles.

The increased efficiency of \AbbrBAB{} algorithm enables its use in practice where we can solve problems with horizon well over a week considering a non-trivial number of jobs quickly, enabling online recalculation of the schedule when energy prices change.
Nevertheless, we acknowledge one of the limitations of our approach, which is that the considered problem statement and our algorithm assume that the energy requirements of the jobs are proportional to their processing time, i.e., the uniform energy consumption.
However, we believe that our approach might be extended to support these considerations---lower bound procedure could be recovered for non-uniform energy consumption~\citep{fang2016singlemach_tou,aghe2019toustates}, but the bigger challenge seems to be designing an efficient primal heuristic such that it frequently finds a solution matching the lower bound, which is the key to the efficiency of \AbbrBAB{}.

We believe that other research may benefit from our results as well.
For example, the idea of exposing the bin packing structure of the problem via the optimal switching problem and lower bound computation, suggesting the spaces that act as bins for the jobs, is particularly useful. 
Considering how well the bin packing problem is solvable in practice, similar problems could be accelerated as well, even when the packing structure of the problem is not exposed at first sight.
Moreover, our idea of incorporating optimal switching costs into the job-interval graph led to improved complexity of the algorithm for the fixed ordering of the jobs, which could be applied to similar problems.

Finally, we outline some ideas for the extension of the problem and the methods proposed.

First, if the jobs have release times and deadlines, these would be reflected by pruning some of the edges in the job-interval graph, but some adjustments should be made to the lower and upper bound computation.
Moreover, it seems to us that our approach could be extended to model machines with multiple dedicated processing states and job families, where each family could be processed only in its dedicated state (e.g., workpieces requiring different hardening temperatures).
These would likely require extending the interval-state graph, but they also would introduce extra edges and vertices into the job-interval graph to model transitions between pairs of different processing states.


\section*{Acknowledgement}
This work was co-funded by the European Union under the project ROBOPROX (reg. no.\\ CZ.02.01.01/00/22\_008/0004590) and the Grant Agency of the Czech Republic under the Project GACR 25-17904S.
Moreover, we would like to thank Theodor Krocan for his insightful comments.

\appendix
\setcounter{equation}{0}\renewcommand\theequation{A\arabic{equation}}
\section*{Appendix}
\renewcommand{\thesubsection}{\Alph{subsection}}

\subsection{Computation of the Optimal State Switching}\label{sec:spaces}

This section describes the problem of pre-computing the optimal switching between the machine states in given intervals while minimizing the power consumption introduced in \cite{2020a:Benedikt}.
After introducing the problem, we summarize the algorithm that computes the optimal switching costs in a polynomial time by finding the shortest paths in the so-called interval-state graph.

\paragraph{Optimal switching problem}
Given the state $\IdxState$ in interval $\Interval{\IdxInterval}$ and state $\IdxAnother{\IdxState}$ in interval $\Interval{\IdxAnother{\IdxInterval}}$ such that $\IdxInterval <  \IdxAnother{\IdxInterval}$, 
the problem finds the optimal transitions from \( (\IdxState, \Interval{\IdxInterval}) \) to \( (\IdxAnother{\IdxState}, \Interval{\IdxAnother{\IdxInterval}}) \) over all possible states $\SetStates$ with respect to the energy cost $\EnergyCostVector$.
Formally, it represents the following optimization problem
\begin{equation} \label{eq:optimal-switching-problem}
    \min_{\SolTrans{\IdxInterval + 1}, \SolTrans{\IdxInterval + 2}, \dots, \SolTrans{\IdxAnother{\IdxInterval} - 1}} \sum_{j = \IdxInterval + 1}^{\IdxAnother{\IdxInterval} - 1} \EnergyCost{j} \cdot \TransPower[\SolTrans{j}].
\end{equation}
such that \( ((\IdxState, \IdxState), \SolTrans{\IdxInterval + 1}, \SolTrans{\IdxInterval + 2}, \dots, \SolTrans{\IdxAnother{\IdxInterval} - 1}, (\IdxAnother{\IdxState}, \IdxAnother{\IdxState}) ) \) form permissible transitions according to transition time function $\TransTime$.
%

\paragraph{Interval-state graph}
It was shown that the optimal switching problem can be solved in polynomial time by finding the shortest path in an \DefTerm{interval-state} graph~\citep{2020a:Benedikt}.
The interval-state graph $G^{\text{is}}$ is defined by a triplet \(G^{\text{is}}= (\SetVertTrans, \SetEdgesTrans, \WeightTrans) \), where \( \SetVertTrans \) is the set of \DefTerm{vertices}, \( \SetEdgesTrans \) is the set of \DefTerm{edges} and \( \WeightTrans: \SetEdgesTrans \rightarrow \mathbb{Z} \) are the \DefTerm{weights} of the edges.
Given time horizon $\NumIntervals$, the sets $\SetVertTrans$ and $\SetEdgesTrans$ are given as: 
\begin{align}
    \SetVertTrans & = \{ \VertTrans{1}{\StateOff}, \VertTrans{\NumIntervals + 1}{\StateOff} \} \cup \{ \VertTrans{\IdxInterval}{\IdxState} : \Interval{\IdxInterval} \in \SetIntervals \setminus \{ \Interval{1} \}, \IdxState \in \SetStates \}, \\
    \begin{split}
        \SetEdgesTrans & = \{ (\VertTrans{1}{\StateOff}, \VertTrans{2}{\StateOff}) \} \cup \{  (\VertTrans{\NumIntervals}{\StateOff}, \VertTrans{\NumIntervals + 1}{\StateOff}) \}  \\
        & \phantom{{}={}} \cup \{ (\VertTrans{\IdxInterval}{\IdxState}, \VertTrans{\IdxInterval + \TransTime[\IdxState,\IdxAnother{\IdxState}]}{\IdxAnother{\IdxState}}) : \IdxState,\IdxAnother{\IdxState} \in \SetStates,  \Interval{\IdxInterval} \in \SetIntervals \setminus \{ \Interval{1} \}, \,\TransTime[\IdxState,\IdxAnother{\IdxState}] \not= \infty, (\IdxInterval - 1) + \TransTime[\IdxState,\IdxAnother{\IdxState}]  \le \NumIntervals - 1 \}\,.
    \end{split}
\end{align}
Each vertex \( \VertTrans{\IdxInterval}{\IdxState} \in \SetVertTrans \) represents that at the beginning of interval \( \Interval{\IdxInterval} \) the machine is in state~\( \IdxState \).
Each edge \( (\VertTrans{\IdxInterval}{\IdxState}, \VertTrans{\IdxAnother{\IdxInterval}}{\IdxAnother{\IdxState}}) \in \SetEdgesTrans \) corresponds to the direct transition from state \( \IdxState \) to state \( \IdxAnother{\IdxState} \) that lasts \( \TransTime[\IdxState,\IdxAnother{\IdxState}] = (\IdxAnother{\IdxInterval} - \IdxInterval) \) intervals.
The condition \( (\IdxInterval - 1) + \TransTime[\IdxState,\IdxAnother{\IdxState}]  \le \NumIntervals - 1 \) ensures, that only transitions completing at most at the beginning of interval \( \Interval{\NumIntervals} \) are present in the interval-state graph.
Weight of edge \( (\VertTrans{\IdxInterval}{\IdxState}, \VertTrans{\IdxAnother{\IdxInterval}}{\IdxAnother{\IdxState}}) \in \SetEdgesTrans \) is defined as
\begin{equation}
    \WeightTrans[\VertTrans{\IdxInterval}{\IdxState}, \VertTrans{\IdxAnother{\IdxInterval}}{\IdxAnother{\IdxState}}] = \sum_{j = \IdxInterval}^{\IdxAnother{\IdxInterval} - 1} \EnergyCost{j} \cdot \TransPower[\IdxState, \IdxAnother{\IdxState}] \,,
\end{equation}
i.e., the total energy cost of the corresponding transition with respect to the costs of energy in intervals.

The optimal transitions from \( (\IdxState, \Interval{\IdxInterval}) \) to \( (\IdxAnother{\IdxState}, \Interval{\IdxAnother{\IdxInterval}}) \) with respect to the energy cost can be obtained by finding the shortest path from \( \VertTrans{\IdxInterval + 1}{\IdxState} \) to \( \VertTrans{\IdxAnother{\IdxInterval}}{\IdxAnother{\IdxState}} \) in the interval-state graph. 
We denote the cost of the optimal switching by function \( \PathCostTrans: \SetVertTrans \times \SetVertTrans \rightarrow  \mathbb{Z} \cup \{\infty\}  \), where the value $\infty$ is used whenever the path between the two vertices does not exist.  
Basically, the interval-state graph is a layered graph that can be seen as a time expansion of the transition time function $\TransTime$ along the horizon defined by intervals $\SetIntervals$, with weights derived from power function $\TransPower$ and energy cost $\EnergyCostVector$.
See \citep{2020a:Benedikt} for an example of an interval-state graph.

\begin{Example}{4}
Consider an interval-state graph representing problem instance from \cref{fig:example-schedule} with a transition graph from Example 1.
For example, the optimal transition of the machine from $\StateProc$ state at interval $\Interval{8}$ to $\StateProc$ state at interval $\Interval{14}$ looks as follows.
At first, the machine is turned off (during $\Interval{9}$), then it remains off (during intervals $\Interval{10}$ and $\Interval{11}$), and finally, it is turned on (intervals $\Interval{12}$, $\Interval{13}$). The cost for this whole path is $\PathCostTrans[\VertTrans{9}{\StateProc}, \VertTrans{14}{\StateProc}] = 76$.


\end{Example}

The values of \( \PathCostTrans \) can be computed using the Floyd-Warshall algorithm that computes all-pairs shortest path between $n$ vertices of the given directed graph in $\AsymCompUpper{n^3}$ steps by iteratively updating the length of the shortest path between every pair of vertices $i$ and $j$, using only vertices $\{1, 2, \ldots, k\}$.
When $k=n$, the algorithm finishes with the optimal values of all-pairs shortest path. 
The interval-state graph $G^{\text{is}}$ has at most $\NumIntervals \cdot \NumStates$ vertices, where $\NumIntervals$ is the length of the horizon (number of considered intervals) and $\NumStates$, is the number of the states of the resource; therefore, the total running time can be upper bounded as \( \AsymCompUpper{\NumIntervals^3 \cdot \NumStates^3} \).


However, only some of the switchings can be considered for scheduling decisions.
Since all the jobs need to be scheduled in the $\StateProc$ state of the machine, the optimal switchings have to be resolved only in (i) between two consecutive intervals with \( \StateProc \); (ii) between the first $\StateOff$ and the first \( \StateProc \); and (iii) the last \( \StateProc \) and the last $\StateOff$. 
The cost of the switchings between \( \IdxState, \IdxAnother{\IdxState} \in \{ \StateOff, \StateProc\}^2\) are specified by function \( \OptCostTrans : \SetIntervals^2 \rightarrow \mathbb{Z} \) defined as 
\begin{equation}
    \OptCostTrans[\IdxInterval, \IdxAnother{\IdxInterval}]=
    \left\lbrace \begin{aligned}
         & \PathCostTrans[\VertTrans{\IdxInterval + 1}{\StateProc}, \VertTrans{\IdxAnother{\IdxInterval}}{\StateProc}] & \IdxInterval > 1, \IdxAnother{\IdxInterval} < \NumIntervals & \quad\quad \text{\footnotesize{case (i)}} \\
         & \PathCostTrans[\VertTrans{2}{\StateOff}, \VertTrans{\IdxAnother{\IdxInterval}}{\StateProc}] & \IdxInterval = 1, \IdxAnother{\IdxInterval} < \NumIntervals & \quad\quad \text{\footnotesize{case (ii)}} \\
         & \PathCostTrans[\VertTrans{\IdxInterval + 1}{\StateProc}, \VertTrans{\NumIntervals}{\StateOff}] & \IdxInterval > 1, \IdxAnother{\IdxInterval} = \NumIntervals & \quad\quad \text{\footnotesize{case (iii)}}
    \end{aligned} \right.
\end{equation}
for each pair of intervals \( \IdxInterval < \IdxAnother{\IdxInterval} \). The vector of states corresponding to $\OptCostTrans[\IdxInterval, \IdxAnother{\IdxInterval}]$, i.e., the \DefTerm{optimal switching behavior} of the machine between $\IdxInterval$ and $\IdxAnother{\IdxInterval}$, is denoted by $\OptPathTrans[\IdxInterval, \IdxAnother{\IdxInterval}]$.

\begin{ExampleCont}{4}
For the illustration, see the \cref{fig:example-schedule}. There, intervals \( \{ \Interval{9}, \Interval{10}, \dots, \Interval{13} \}\) represent the space between two jobs \( \Job{2}, \Job{1} \) with cost \( \OptCostTrans[8,14] = \PathCostTrans[\VertTrans{9}{\StateProc}, \VertTrans{14}{\StateProc}] = 76 \).

The optimal switching behavior is
\begin{equation}
    \OptPathTrans[8,14] = ((\StateProc, \StateOff), (\StateOff, \StateOff), (\StateOff, \StateOff), (\StateOff,\StateProc), (\StateOff,\StateProc)).
\end{equation}
\end{ExampleCont}

Values of \( \OptCostTrans \) can be computed efficiently using the algorithm called the \DefTerm{Shortest Path Algorithm for Cost Efficient Switchings} (\AbbrIwa{})~\citep{2020a:Benedikt}.
For each  value \( \Interval{\IdxInterval} \in \SetIntervals \setminus \{ \Interval{\NumIntervals} \} \), \AbbrIwa{} computes all values \( \OptCostTrans[\IdxInterval, \IdxInterval + 1], \OptCostTrans[\IdxInterval, \IdxInterval + 2], \dots, \OptCostTrans[\IdxInterval, \NumIntervals] \) by computing the shortest paths tree from \( \VertTrans{\IdxInterval + 1}{\StateProc} \) (or \( \VertTrans{2}{\StateOff} \) if \( \IdxInterval = 1 \)) to all other vertices in the interval-state graph.
If all weights $\WeightTrans$ happen to be nonnegative (e.g., when energy cost $\EnergyCostVector\geq\boldsymbol{0}$), then the shortest path tree is obtained with the Dijkstra algorithm that runs in \( \AsymCompUpper{|\SetEdgesTrans| + |\SetVertTrans| \cdot \log |\SetVertTrans|}\).
Since the Dijkstra algorithm is started \( \NumIntervals \) times, the total complexity is 
\begin{equation}
    \AsymCompUpper{\NumIntervals \cdot ( |\SetEdgesTrans| + |\SetVertTrans| \cdot \log |\SetVertTrans|)} = \AsymCompUpper{\NumIntervals^2 \cdot \NumStates \cdot (\NumStates + \log \NumIntervals +\log \NumStates)}\,.
\end{equation}
In case of the presence of negative weights $\WeightTrans$, one has to resort to the Bellman-Ford algorithm with complexity $\AsymCompUpper{|\SetVertTrans|\cdot |\SetEdgesTrans|}$, thus in total, the complexity would be $\AsymCompUpper{\NumIntervals^3\cdot|\SetStates|^3}$.



\bibliography{references}

\end{document}